\documentclass[reqno]{amsart}

\address{Department of Mathematics, Kyoto
University, Kyoto, Japan } \email{fukaya@math.kyoto-u.ac.jp}
\address{Department of Mathematics, University of
Wisconsin, Madison, WI, USA \& \newline \indent National Institute
for Mathematical Sciences, Daeduk Boulevard 628, Yuseong-gu,
Daejeon, 305-340, Korea } \email{oh@math.wisc.edu}
\address{Graduate School of Mathematics,
Nagoya University, Nagoya, Japan } \email{ohta@math.nagoya-u.ac.jp}
\address{Department of Mathematics,
Hokkaido University, Sapporo, Japan }
\email{ono@math.sci.hokudai.ac.jp}

\usepackage{graphicx}
\usepackage{amscd}
\usepackage{amssymb}
\usepackage{amstext}
\usepackage{amsmath}
\usepackage{enumerate}
\usepackage[all]{xy}
\newtheorem{theorem}{Theorem}[section] 
\newtheorem{lem}[theorem]{Lemma}     
\newtheorem{sublem}[theorem]{Sublemma}
\newtheorem{cor}[theorem]{Corollary}
\newtheorem{prop}[theorem]{Proposition}
\theoremstyle{definition}
\newtheorem{defn}{Definition}[section]

\newtheorem{rem}[defn]{Remark}

\newtheorem{ass}[defn]{Assumption}

\def\R{{\mathbb R}}
\def\C{{\mathbb C}}
\def\Z{{\mathbb Z}}
\def\Q{{\mathbb Q}}

\def\P{{\mathbb P}}
\def\Sph{{\mathbf S}}

\def\e{{\epsilon}}

\def\MM{{\mathcal M}}

\def\SS{{\mathcal S}}

\title[Non-displaceable Lagrangian tori in $S^2 \times S^2$]
{Toric degeneration and non-displaceable Lagrangian tori in $S^2
\times S^2$}

\author{Kenji FUKAYA, Yong Geun OH, \\
Hiroshi OHTA, and Kaoru ONO}

\thanks{KF is supported partially by JSPS Grant-in-Aid for Scientific Research
No.18104001 and Global COE Program G08, YO by US NSF grant \#
0904197, HO by JSPS Grant-in-Aid for Scientific Research
No.19340017, and KO by JSPS Grant-in-Aid for Scientific Research,
Nos. 21244002.}
\date{February 8, 2010}
\begin{document}
\begin{abstract} In this article, using the idea of toric degeneration
and the computation of the full potential function of Hirzebruch
surface $F_2$, which is \emph{not} Fano, we produce a continuum of
Lagrangian tori in $S^2 \times S^2$ which are non-displaceable under
the Hamiltonian isotopy.
\end{abstract}
\maketitle

\section{Introduction}\label{intro}
In \cite{toric1,toric2}, we study Lagrangian Floer theory of
Lagrangian torus fibers in toric manifolds, in particular, that of
non-displaceable under the Hamiltonian isotopy.
In this article, we discuss non-displaceable Lagrangian tori in
$(S^2, \omega_{\rm std}) \times (S^2, \omega_{\rm std})$,
which are \emph{not} of the type of toric fibers.
Here $\omega_{\rm std}$ denotes the symplectic form on $S^2$ with area $2\pi$.
Our main result is the following.

\begin{theorem}\label{maintheorem}
There exist uncountably many Lagrangian tori $T(u)$ in
$(S^2, \omega_{\rm std}) \times (S^2, \omega_{\rm std})$,
parameterized by the real numbers $u \in (0,u_0]$ for some $u_0 > 0$, such that:
\begin{enumerate}
\item If $u \ne u'$ then  $T(u)$ is not Hamiltonian
isotopic to $T(u')$.
\item For any $u$ there exists a pair $(\frak b,b)
\in H^2(S^2 \times S^2;\Lambda_+) \times
H^1(T(u);\Lambda_0)$ such that the Floer cohomology
$HF((T(u),(\frak b,b)),(T(u),(\frak b,b)))$ is isomorphic to $H(T(u);\Lambda_0)$.
In particular none of them are displaceable.
\item $T(u) \cap T(u') = \emptyset$ if $u\ne u'$.
\item There exists a unique $T(u_0)$ in our family
that is monotone.
\item $T(u_0)$ is not symplectically equivalent to
$S^1_{\text{\rm eq}} \times S^1_{\text{\rm eq}}$, the direct product of the
equators.
\end{enumerate}
\end{theorem}
The definitions of the Floer cohomology $HF((T(u),(\frak b,b)),(T(u),(\frak b,b)))$ with
bulk deformation is
given in \cite{fooo-book} section 3.8.

We will give an explicit description of $T(u)$ in later sections.
Using either the theory of spectral invariants with bulk deformation or
a generation result for Fukaya category of toric manifolds, we can obtain the
intersection result $\varphi(T(u)) \cap (S^1_{\text{\rm eq}} \times S^1_{\text{\rm eq}}) \ne \emptyset$
for any Hamiltonian diffeomorphism $\varphi$. (See Remark \ref{int-result}.)

In response to Polterovich's question, Albers and Frauenfelder \cite{A-F} proved Hamiltonian
non-displaceability of a certain monotone
Lagrangian torus in $T^*S^2$ with the standard symplectic structure.
Our theorem also implies this non-displaceability result. (See Remark \ref{A-F}.)
\par
We use the universal Novikov ring $\Lambda_{0,\text{\rm nov}}$ and
its ideal $\Lambda_{0,\text{\rm nov}}^+$ in this paper. We recall their definitions here.
An element of $\Lambda_{0,\text{\rm nov}}$ is a formal sum
$\sum a_iT^{\lambda_i}e^{\mu_i}$ with $a_i \in \C$, $\lambda_i \in \R_{\ge 0}$,
$\mu_i \in \Z$ such that $\lambda_i \leq \lambda_{i+1}$ and $\lim_{i\to \infty} \lambda_i = \infty$.
$T$ and $e$ are formal parameters. We define
a valuation $\frak v_T: \Lambda_{0,\text{\rm nov}} \to \R_{\ge 0}$ defined by
$$
\frak v_T\left(\sum_{i=1}^\infty a_iT^{\lambda_i}e^{\mu_i}\right) = \lambda_1.
$$
This induces a natural $\R$-filtration on $\Lambda_{0,\text{\rm nov}}$
which in turn induces a non-Archimedean topology thereon.
The sum is said to be an element of $\Lambda_{0,\text{\rm nov}}^+$
if $\lambda_i > 0$ for all $i$. If we rearrange the sum
$\sum a_iT^{\lambda_i}e^{\mu_i}$ into
$$
\sum_{k=1}^\infty p_k(e) T^{\lambda_{i_k}}, \quad \lambda_{i_k} < \lambda_{i_{k+1}},
$$
then each $p_k$ becomes a complex polynomial
of variables $e, \, e^{-1}$. In particular, we can insert $e = 1$ and the resulting formal sum
converges in non-Archimedean topology and satisfies
$$
\frak v_T\left(\sum_{k=1}^\infty p_k(1) T^{\lambda_{i_k}}\right) \geq
\frak v_T\left(\sum_{k=1}^\infty p_k(e) T^{\lambda_{i_k}}\right) = \lambda_1.
$$
(We note that the value $p_1(1)$ could be zero.)
We also use the subring $\Lambda_0$ of $\Lambda_{0,\text{\rm nov}}$
consisting of elements which do not involve $e$: We define
$$
\Lambda_0 = \left\{\sum_{i=1} a_i T^{\lambda_i} \, \Big\vert\,  \lambda_i \le \lambda_{i+1}, \,
 \lim_{i \to \infty} \lambda_i = \infty \right\}
$$
which is isomorphic to $\Lambda_{0,\text{\rm nov}}|_{\{e=1\}}$, the quotient of
$\Lambda_{0,\text{\rm nov}}$ by the ideal generated by $e-1$,
and its unique maximal ideal by
$
\Lambda_+ = \Lambda_0 \cap \Lambda_{0,\text{\rm nov}}^+.
$
One can unify the definitions of $\Lambda_{0,\text{\rm nov}}$ and of $\Lambda_0$ by
introducing a universal Novikov ring $\Lambda_0^R$ over a general coefficient ring $R$
defined by
$$
\Lambda_0^R : =
\left\{\sum_{i=1} a_i T^{\lambda_i} \, \Big\vert\,  \lambda_i \le \lambda_{i+1}, \,
a_i \in R, \, \lim_{i \to \infty} \lambda_i = \infty \right\}
$$
Then we can write the above universal Novikov rings as
$$
\Lambda_{0,\text{\rm nov}} = \Lambda_0^{\C[e,e^{-1}]}, \quad \Lambda_0 = \Lambda_0^\C
$$
for $R = \C[e, e^{-1}], \, \C$ respectively.
\section{Potential function of the Hirzebruch surface $F_2(\alpha)$}\label{Hirz}
We consider toric Hirzebruch surface $F_2(\alpha)$ whose
moment polytope is
\begin{equation}\label{POord}
P(\alpha) =
\{
(u_1,u_2) \in \R^2 \mid u_i \ge 0,
u_2 \le 1-\alpha, u_1+2u_2 \le 2
\}
\end{equation}
Recall that the fiber $L({\mathbf u}) = \pi^{-1}({\mathbf u})$ at ${\mathbf u} \in \text{\rm Int }P(\alpha)$ is a
Lagrangian torus. We fix an identification of $L({\mathbf u}) \cong T^2$ and
an integral basis $\{{\bf e}_i^*\}_{i=1,2}$ of $H_1(L({\mathbf u});\Z) \cong H_1(T^2;\Z) \cong \Z^2$
and its dual basis $\{{\bf e}_i\}$ on $H^1(L({\mathbf u});\Z)$. We denote by $\{x_i\}_{i=1,2}$
the coordinates of $H^1(L({\mathbf u});\Lambda_0)$ with respect to $\{{\bf e}_i\}$, and set $y_i = e^{x_i}$.

$F_2(\alpha)$ is not Fano but nef, i.e. every holomorphic sphere has non-negative
Chern number. In fact the toric divisor $D_1 \cong \C P^1$ associated to the facet of
$P(\alpha)$, $P_1 = \{ u \in \partial P(\alpha) \mid u_2 = 1-\alpha\}$ has
$c_1(D_1) = 0$.  Denote by $D_2,D_3,D_4$ remaining toric divisors of $F_2(\alpha)$.

Let $\beta_i \in H_2(F_2(\alpha),L({\mathbf u});\Z)$ ($i=1,\dots,4$) be the classes such that
$\beta_i \cap D_j = \delta_{ij}$.  Then we have
$\mathcal M_1^{\text{\rm reg}}(F_2(\alpha),L({\mathbf u});\beta_i) \ne \emptyset$
and $c_1 \cap \beta_i = 2$. (Here `reg' means the moduli space of pseudo-holomorphic disks without
bubble.) \cite{cho-oh} Theorem 5.2 implies that there are exactly four homology classes satisfying this
condition. $\beta_1$ is the same class as before.

We refer to \cite{fooo-book} for the general definition of the potential function $\frak{PO}$.
We also provide some description of $\frak{PO}$ specialized to the
current circumstance in Appendix of the present paper. The following description of the potential function $\frak{PO}$
of $F_2(\alpha)$ can be derived from the results from \cite{toric1}
Example 8.2 and the argument used in the proof of \cite{toric2} Proposition 9.4.
For readers' convenience, we give its proof.

\begin{prop}\label{F2PO}
The potential function $\frak{PO}$ of $F_2(\alpha)$ has the form
\begin{equation}\label{POord2}
\frak{PO}
=\frak{PO}(y_1,y_2;u_1, u_2) =
T^{u_1}y_1 + T^{u_2}y_2 +T^{2-u_1-2u_2}y_1^{-1}y_2^{-2} +
(1+c)T^{1-\alpha -u_2}y_2^{-1}
\end{equation}
for some element $c \in \Lambda_+$ of the form
\begin{equation}\label{eq:c}
c = \sum_{k\ge 1}c_k T^{2k\alpha}, \quad c_k \in \Q.
\end{equation}
\end{prop}
\begin{proof} In our current circumstance, we have
$$
\frak{PO}(b) =
\sum_{\mu(\beta) =2} T^{\omega(\beta)/2\pi} \exp (b \cap \partial \beta)\,
\text{\rm deg}[ev_{0*}:\mathcal M_1(\beta) \to L({\mathbf u})].
$$
(See Theorem \ref{POcanwithb} and \eqref{PO^L} in Appendix. Here we divide
$\omega(\beta )$ in the exponent of $T$ by $2\pi$ to be consistent with
the conventions used in \cite{toric1}, \cite{toric2}.)
We decompose
$$
\frak{PO} = \frak{PO}_0 + \text{\rm ``higher order part"}
$$
where $\frak{PO}_0$
is the `leading order part' of $\frak{PO}$ coming from the contribution of
the classes $\beta_i$, $i=1, \ldots, 4$.
We have derived in \cite{toric1} Example 8.2
$$
\frak{PO}_0 = T^{u_1}y_1 + T^{u_2}y_2 +T^{2-u_1-2u_2}y_1^{-1}y_2^{-2} +
T^{1-\alpha -u_2}y_2^{-1}.
$$
It remains to identify the ``higher order part''. This is the contribution
of the singular discs in classes $\beta$ with $\mu(\beta) = 2$.

Let $\beta \in H_2(X,L({\mathbf u});\Z)$ with $\mu(\beta) = 2$.
We assume $\mathcal M_{1;\ell}^{\text{\rm main}}(L({\mathbf u}),\beta)$ is nonempty. According
to \cite{toric1} Theorem 11.1 (5) combined with the fact $F_2(\alpha)$ being nef,
$\beta$ must be of the form
$$
\beta = \beta_i + k [D_1]
$$
for some $i = 1, \ldots, 4$.
(See also \cite{toric2} Proposition 10.4.) On the other hand, the bordered stable map must have connected image, and that
only the holomorphic discs in $\beta_1$ among $\beta_i$'s intersect the divisor $D_1$,
which follows from the classification theorem \cite{cho-oh} Theorem 5.2
Therefore $\beta$ must be of the form $\beta_1 + k [D_1]$ for $k \in \Z_{\ge 0}$.

Writing $b = x_1 {\bf e}_1 + x_2 {\bf e}_2$ and noting ${\bf e}_1 \cap \partial \beta_1 = 0$,
${\bf e}_2 \cap \partial \beta_1 = -1$
, we obtain
$$
\exp (b \cap \partial \beta_1) = \exp \left((x_1 {\bf e}_1 + x_2 {\bf e}_2) \cap \partial \beta_1)\right)
= e^{x_1({\bf e}_1 \cap \partial \beta_1)} e^{x_2({\bf e}_2 \cap \partial \beta_1)} = y_2^{-1}.
$$
Furthermore
$$
\omega(\beta) = \omega(\beta_1) + k \omega([D_1]) = 2\pi((1-\alpha - u_2) + 2k\alpha)
$$
since we have $\omega(\beta_1) = 2\pi(1-\alpha - u_2)$ (see \cite{cho-oh} Theorem 8.1).
Therefore $\beta_1 + k [D_1]$ contributes
$$
c_k T^{2k\alpha} T^{1-\alpha - u_2} y_2^{-1}
$$
to $\frak{PO}^u$ where $c_k$ is given by
\begin{equation}\label{ckdegree}
c_k = \text{\rm deg}\left[ev_0 : \mathcal M_1(F_2(\alpha),L({\mathbf u});\beta_1 + k\,[D_1]) \to
L({\mathbf u})\right].
\end{equation}
(We remark that the symplectic area of the $(-2)$-curve  $D_1$ is $2\pi(2\alpha)$.)
By summing over $k$'s, we obtain the proposition.
\end{proof}
 Appearance of the additional term
$cT^{1-\alpha -u_2}y_2^{-1}$ in $\frak{PO}$ reflects the fact that $F_2(\alpha)$
in \emph{not} Fano.

The following theorem completely determines the full potential $\frak{PO}$
for this non-Fano toric manifold $F_2(\alpha)$, which will play an important role
in our study of non-displaceable tori later.

\begin{theorem}\label{hirzthem} We have $c = T^{2\alpha}$.
In particular $c_k = 0$ for $k \geq 2$ and $c_1 = 1$.
\end{theorem}
\begin{rem}
This result, in the form of convergent power series, is obtained previously
by D. Auroux \cite{auroux} using a different method.
Although Auroux \cite{auroux} did not state the version with coefficients in
the Novikov ring, he determined all necessary information on the moduli spaces
of bordered stable maps in order to determine the potential function by analyzing the
wall-crossing phenomenon.  Hence Theorem \ref{hirzthem} follows from his study.
\end{rem}
\section{Smoothing of singular toric $F_2(0)$}\label{Deformation}
We consider the limit $\alpha \to 0$ of our Hirzebruch surface $F_2(\alpha)$.
At $\alpha =0$ we obtain an orbifold with a singularity of the form
$\C^2/\{\pm 1\}$. This singularity is of $A_1$-type:
The map
$$
(x,y) \mapsto (x^2,y^2,xy)
$$
induces an isomorphism between $\C^2/\{\pm 1\}$ and
$
\{u,v,w) \mid uv = w^2\}.
$
The link of the singular point is diffeomorphic to $S^3/\{\pm 1\} = \R P^3$.
\par
We deform the latter singular surface to a Milnor fiber
$\{(u,v,w) \mid uv = w^2 + \epsilon^2 \}$.
Here we have $\epsilon^2$ in the right hand side in order to obtain the simultaneous resolution of
this family.
We cut out a neighborhood of
the singularity of $F_2(0)$ and paste the Milnor fiber back into
the neighborhood to obtain the desired manifold. We denote it by $\widehat F_2(0)$.
\par

We also have a symplectic description of $\widehat F_2(0)$ which is in order.
Consider the preimage $Y(\varepsilon)$ of $P(\varepsilon) \subset P(0)$, $0<\varepsilon <1$,
under the moment map.  Then $Y(\varepsilon)$ has concave boundary $\partial Y(\varepsilon)$
which is diffeomorphic to $S^3/\{\pm 1\}$.  Moreover, the characteristic foliation of $\partial Y(\varepsilon )$
is same as that of the contact manifold
$S^3/\{\pm 1\}$ equipped with the standard contact form $\theta_{\text{\rm can}}$
whose leaves consist of the fibers of the circle bundle $\R P^3 = S^3/\{\pm 1\} \to S^2$.
\par
We also note that the Milnor fiber of $\C^2 /\{\pm 1\}$ is diffeomorphic to $T^*S^2$.
Let $\Sph ^2$ be the standard round 2-sphere and $D_{r}(T^*\Sph ^2)$
its cotangent disc bundle with radius $r>0$. $D_{r}(T^* \Sph ^2)$ has
convex boundary $\partial D_{r}(T^* \Sph ^2)$ which is diffeomorphic to $S^3/\{\pm 1\}$.
The characteristic foliation of $\partial D_{r}(T^* \Sph ^2)$
is isomorphic to that of $(S^3/\{\pm 1\},\theta_{\text{\rm can}})$.
Hence we can take a suitable radius $r=r(\varepsilon)>0$ so that the symplectic form on
a collar neighborhood $N(\varepsilon)$ of $\partial Y(\varepsilon)$ and the one on a collar neighborhood of
$\partial D_r(T^* \Sph ^2)$ can be glued to a symplectic form on $\widehat F_2(0) =
Y(\varepsilon) \cup D_r(T^* \Sph ^2)$ in a way that the given toric
symplectic form on $Y(\varepsilon)$ is unchanged on
\begin{equation}\label{YeminusNe}
Y(\varepsilon) \setminus N(\varepsilon) \subset Y(\varepsilon) \setminus \partial Y(\varepsilon).
\end{equation}
Since $H^2(S^3/\{\pm 1\};\Q) = 0$, the glued symplectic form does not depend on the choices
of $\varepsilon> 0$ or the gluing data up to the symplectic diffeomorphism.
\begin{rem}\label{A-F}
Note that the projections to $S^2$ of characteristic Reeb orbits in
$\partial D_{r}(T^* \Sph ^2)$ are oriented geodesics on $\Sph ^2$,
all of which are periodic. Hence the space of (unparameterized)
oriented geodesics with a minimal period, denoted by $\text{\rm
Geod}^+_1(\Sph ^2)$, is identified with the Grassmannian of oriented
$2$-planes in $\R ^3$ which is diffeomorphic to $S^2$. The
Lagrangian torus $L_N$ studied in \cite{A-F} is the union of closed
oriented geodesics of unit speed with a minimal period passing
through a given point, say the north pole $N$. Denote by $\sigma_c$
the fiberwise multiplication by $c>0$ on $T^* \Sph ^2$. We claim
that $L_N$ is non-displaceable. Suppose to the contrary that there
is a Hamiltonian diffeomorphism $\phi$ of $T^* \Sph ^2$ with $\phi
(L_N) \cap L_N = \emptyset$. Then we can take a sufficiently small
$c>0$ such that $\sigma_c (\text{\rm supp}(\phi)) \subset
D_{r/2}(T^* \Sph ^2)$. Therefore $\sigma_c(L_N)$ can be regarded as
a Lagrangian torus contained in $\widehat{F}_2(0)$ and $\sigma_c
\circ \phi \circ \sigma_{c^{-1}}$ is a Hamiltonian diffeomorphism of
$\widehat{F}_2(0)$, which disjoin $\sigma_c(L_N)$. However this
Lagrangian torus $L'$ is the same as the inverse image of a great
circle in $S^2$ by the projection  $\partial D_{r}(T^* \Sph ^2) \to
S^2=\text{\rm Geod}^+_1(\Sph ^2)$. By taking a sufficiently small
constant $\varepsilon' > 0$, we find that $L'$ is contained in
$Y(\varepsilon') \setminus N(\varepsilon') \subset \widehat{F}_2(0)$
and becomes one of the Lagrangian torus $T(u)$ given by the
construction in section 4. This then gives rise to a contradiction
to Theorem 1.1 (2) and hence $L_N$ is not displaceable in $T^* \Sph
^2$.
\end{rem}

\begin{lem}\label{samearea}
$\widehat F_2(0)$ is symplectomorphic to
$(S^2,\omega_{\rm std}) \times (S^2,\omega_{\rm std})$.
\end{lem}
This is a well-known fact. We will prove a stronger fact, Proposition \ref{hireqchokuseki} in section \ref{proof1},
which contains Lemma \ref{samearea} as a special case.
For ${\mathbf u} \in \text{\rm Int}\,P(0)$, we choose $\varepsilon > 0$ small enough such that
${\mathbf u} \in P(\varepsilon )$.
Then we find that $\widehat F_2(0)$ still contains $L({\mathbf u})$.
When $L({\mathbf u})$ with ${\mathbf u}=(u,1-u)$ is considered
as a Lagrangian torus in $\widehat{F}_2(0)$, we denote it by $T(u)$.
We will give more precise description in section 4.
The contraction of a vanishing cycle $\cong S^2$ in $\widehat F_2(0)$ gives
a map
$$
\widehat \pi: \widehat F_2(0) \to F_2(0).
$$
This map induces a
canonical commutative diagram in homology
$$
\xymatrix{ {} \ar[r] &H_2(\widehat F_2(0);\Z) \ar[r]^(.4){i_*} \ar[d]^{\widehat \pi_*}&  H_2(\widehat F_2(0),T(u);\Z)
\ar[r]^{\delta} \ar[d]^{\widehat \pi_*} &H_1(T(u);\Z) \ar[d]^{\widehat \pi_*} \ar[r] & \\
{} \ar[r] &H_2( F_2(0);\Z) \ar[r]^(.4){i_*}&  H_2(F_2(0),L(u,1-u);\Z)
\ar[r]^{\delta} & H_1(L(u,1-u);\Z) \ar[r] &}
$$
Here $\widehat \pi_*: H_1(T(u);\Z) \to H_1(L(u,1-u);\Z)$ is an isomorphism which
becomes the identity map under the above identification. We fix a basis of
$H_1(L({\mathbf u});\Z)$,
let $\{x_i\}_{i=1,2}$ be the coordinates of $H^1(L({\mathbf u});\Lambda_0)$ with respect to
its dual basis and set $y_i = e^{x_i}$ as before.

\begin{theorem}\label{POsingptt}
Let $\frak{PO}$ be the potential function of $\widehat F_2(0)$ written in terms of
the above mentioned basis. Then we have
\begin{equation}\label{POsingpt}
\frak{PO}
=T^{u_1}y_1 + T^{u_2}y_2 +T^{2-u_1-2u_2}y_1^{-1}y_2^{-2} +
2T^{1 -u_2}y_2^{-1}.
\end{equation}
\end{theorem}
We note that (\ref{POsingpt}) can be obtained by putting $\alpha = 0$ in (\ref{POord2}).
However to justify this conclusion, we need some more work to do.
\begin{rem}
The idea of using degeneration to a singular toric variety in a calculation of
the potential function for a non toric manifold is due to
Nishinou-Nohara-Ueda \cite{nnu1,nnu2}. The transition $F_2(\alpha) \to F_2(0) \to
\widehat F_2(0)$ is a baby example of \emph{conifold transition} in physics literature,
the resolution $F_2(0)$ to $\widehat F_2(0)$ is a baby example of \emph{crepant resolution}
and the degeneration of $\widehat F_2(0)$ to $F_2(0)$ is an example of \emph{toric degeneration}
of a non-toric $\widehat F_2(0)$.
\end{rem}
Using the formula \eqref{POsingpt} in Theorem \ref{POsingptt}, we now find critical points of the potential function
$\frak{PO}$ at the point $u =(1/2,1/2)$ for $\widehat F_2(0)$. Note $L(1/2,1/2)$ is a
monotone Lagrangian torus in $\widehat F_2(0)$. We have
\begin{equation}\label{POsingptmonotone}
\frak{PO}^u
= T^{1/2}(y_1 + y_2 +y_1^{-1}y_2^{-2} +
2y_2^{-1}).
\end{equation}
The critical point equation of $\frak{PO}^u$ for $(y_1,y_2)$ becomes
\begin{eqnarray}
0 &=& 1 - y_1^{-2}y_2^{-2}.  \label{cond1}\\
0 &=& 1 -2 y_1^{-1}y_2^{-3} - 2y_2^{-2}.\label{cond2}
\end{eqnarray}
The first equation (\ref{cond1}) implies $y_1y_2 = \pm 1$.
\par\smallskip
\emph{Case 1}; $y_1y_2 = -1$.
The second equation (\ref{cond2}) becomes $1=0$, which is absurd.
\par
\emph{Case 2}; $y_1y_2 = 1$. The equation (\ref{cond2}) is equivalent to $y_2 = \pm 2$.
Therefore we conclude that there are 2 solutions for the critical point equation
at ${\mathbf u} = (1/2,1/2)$.
\par
On the other hand, we remark that for $\alpha > 0$,
there is a unique balanced fiber ${\mathbf u} = ((1+\alpha)/2,(1-\alpha)/2)$
which carries 4 critical points. So the valuation
$(\frak v_T(y_1), \frak v_T(y_2))$
of 2 critical points among those 4,
which is nothing but the location of the fiber ${\mathbf u} = (u_1, u_2)$
(see \cite{toric1} section 7 for a detailed explanation),
jumps away from $(1/2,1/2)$ to somewhere else.
In fact, they jump to the point ${\mathbf u}=(0,1)$ in the following sense:
The valuation point ${\mathbf u} = (0,1)$ corresponds to a singular
point of $F_2(0)$ which no longer carries a torus action.
However there appears a new Lagrangian sphere $S^2$ in its smoothing
$\widehat F_2(0)$. We may regard that this Lagrangian sphere corresponds
to the two missing critical points. \emph{The two torus branes are merged and
transformed into a sphere brane under a conifold transition!}

\begin{rem}
For the case of the two-point blow-up of $\C P^2$ we have
uncountably many non-displaceable $T^2$ at the
very moment when the location of balanced fibers jumps.
Their location lies on the line segment joining
the positions of balanced fibers before and after the jump.
(See \cite{toric2} section 5 for such an example.)
The same phenomenon occurs here.
\end{rem}

We consider the tori $L({\mathbf u})$ in $F_2(0)\setminus \pi^{-1}((0,1))$
at ${\mathbf u} = (u_1,u_2)$ on the line segment given by
$$
u_1 = 2-u_1 - 2 u_2 = 1-u_2 < u_2.
$$
It can be considered as a submanifold of $\widehat F_2(0)$ and
we denote it by $T(u) \subset \widehat{F}_2(0)$.
This is possible by the discussion
around \eqref{YeminusNe}
if we choose $\varepsilon$ sufficiently small relative to the distance from
$\mathbf u$ to $(0,1)$. The above equation is equivalent to
\begin{equation}\label{ulines}
u_1 = 1-u_2 ,\quad u_2 > 1/2.
\end{equation}
On this line, the leading order term of the potential function is
$$
y_1 + y_1^{-1}y_2^{-2} + 2y_2^{-1}.
$$
Therefore the leading term equation introduced in \cite{toric1}
Definition 10.2 is reduced to
\begin{eqnarray}
0 &=& 1 - y_1^{-2}y_2^{-2}.  \label{lte1}\\
0 &=& -2 y_1^{-1}y_2^{-3} - 2y_2^{-2}.\label{lte2}
\end{eqnarray}
The equation (\ref{lte1}) is equivalent to $y_1y_2 =\pm 1$.
In case $y_1y_2 =1$, (\ref{lte2}) has no solution.
However in case $y_1y_2 =-1$, (\ref{lte2}) becomes
vacuous. Therefore the leading term equation has a continuum of solutions
which will give rise to the continuum of Lagrangian tori $T(u)$
mentioned in Theorem \ref{maintheorem}.
\par
We would like to note that the Lagrangian torus $T(u)$ in Theorem \ref{maintheorem}
is \emph{not} a torus fiber \emph{of a toric manifold}. Because of this,
we can not directly apply \cite{toric2} Theorem 1.3 to conclude
non-vanishing of a Floer cohomology associated to $T(u)$.
However we can still prove the following.
(See \cite{fooo-book} section 3.8 for the notations appearing in
Theorem \ref{nonvanish}.)
\begin{theorem}\label{nonvanish}
Let $M=(S^2,\omega_{\rm std}) \times (S^2,\omega_{\rm std})$.
If $(u_1,u_2)$ satisfies (\ref{ulines}) then there exist a bounding cochain with bulk $(\frak b,b)$ of
$\frak b \in H^2(M;\Lambda_+)$ and a bounding cochain $b \in H^1(T(u);\Lambda_0)$ such that
$HF((T(u),(\frak b,b)),(T(u),(\frak b,b))) \ne 0$. In addition $\frak b$ has the property
that it carries exactly two elements $b \in H^1(T(u);\Lambda_0)$ such that
$HF((T(u),(\frak b,b)),(T(u),(\frak b,b))) \ne 0$.
\end{theorem}
\begin{proof}{}[Proof of Theorem \ref{nonvanish} $\Rightarrow$ Theorem \ref{maintheorem}]
Statement (2) is Theorem \ref{nonvanish}.  (3) is immediate from construction.
(1) follows from (2), (3) and the invariance property of Floer cohomology
under the action of Hamiltonian diffeomorphism.
(4) follows from the fact that only $T(1/2)$ is monotone among $T(u)$'s.
Finally we prove (5). The last paragraph of Theorem \ref{nonvanish} states
that there exists $\frak b$ for which the number of bounding cochain $b \in H^1(T(u);\Lambda_0)$ with
$HF((T(u),(\frak b,b)),(T(u),(\frak b,b))) \ne 0$ is
exactly $2$. On the other hand, we know from \cite{toric2}
that for any choice of $\frak b$ there exist
exactly $4$ different choices of $b \in H^1(S^1_{\text{\rm eq}} \times S^1_{\text{\rm eq}})$ such that
$$
HF((S^1_{\text{\rm eq}} \times S^1_{\text{\rm eq}},(\frak b,b)),
(S^1_{\text{\rm eq}} \times S^1_{\text{\rm eq}},(\frak b,b))) \ne 0.
$$
We recall from section 4.3.3 in \cite{fooo-book} that
\begin{enumerate}
\item the potential function
$\frak{PO}^{\frak b}: \widehat \MM_{\rm weak}(L;\frak m^{\frak b}) \to \Lambda_0$ is gauge invariant
\item the isomorphism type of Floer cohomology depends only on the gauge
equivalence class in the space $\widehat \MM_{\rm weak}(L;\frak m^{\frak b})$ of weak
bounding cochains.
\end{enumerate}
In particular the set of weak bounding cochains $b$
with nontrivial Floer cohomology is symplectically invariant.
(In this paper, we suppress the formal variable ``$e$'' and work with $\Z/2\Z$-grading.)
In order to extend the coefficients to $\Lambda_0$, we use non-unitary flat line bundle
on the Lagrangian submanifold as in section 12 of \cite{toric1} and adapt the argument
in section 4.3.3 in \cite{fooo-book} accordingly.  The idea of
using non-unitary flat line bundle is originally due to Cho \cite{Cho}.

Hence the proof of (5) is reduced to the following lemma.
\begin{lem} Let $L$ be either $T(u)$ or $S^1_{\text{\rm eq}} \times S^1_{\text{\rm eq}}$
in $M=S^2(1) \times S^2(1)$. Then we have a canonical isomorphism
$$
H^1(L;\Lambda_0) \cong \MM_{\rm weak}(L;\frak m^{\frak b})
$$
and hence we can naturally identify $\MM_{\rm weak}(L;\frak m^{\frak b})$ with $H^1(L;\Lambda_0)$
and $\frak{PO}^{\frak b}$ can be regarded as a function defined on $H^1(L;\Lambda_0)$
for all $\frak b$.
\end{lem}
\begin{proof}
In our two dimensional case, all the non-constant pseudo-holomorphic discs,
which have boundary on $L= L(u)$ or $S^1_{\rm eq} \times S^1_{\rm eq}$,  have Maslov indices
at least 2,  if the almost complex is generic.
Therefore we can use the argument in \cite{toric1} and \cite{toric2} to show:
$$
H^1(L;\Lambda_0) \hookrightarrow \widehat \MM_{\rm weak}(L;\frak m^{{\rm can},{\frak b}})
$$
and the gauge transformation acts  trivially on its image.   (See section \ref{secapend}.)
Therefore we have an embedding
$$
H^1(L;\Lambda_0) \hookrightarrow \MM_{\rm weak}(L;\frak m^{{\rm can}, {\frak b}}).
$$
And we may assume
$$
\widehat \MM_{\rm weak}(L;\frak m^{{\rm can}, {\frak b}}) \subset
C^{\rm odd}_{\rm dR}(L;\Lambda_0)
$$
by considering the de Rham model. We can also identify $\MM_{\rm weak}(L;\frak m^{{\rm can}, {\frak b}})$ as a sub-variety of
$H^{\rm odd}(L;\Lambda_0)$ as explained in section 5.4 \cite{fooo-book}.
But since both $L$ are of 2 dimension, we have
$H^1(L;\Lambda_0) = H^{\rm odd}(L;\Lambda_0)$
and hence $H^1(L;\Lambda) \cong \MM_{\rm weak}(L;\frak m^{{\rm can},{\frak b}})$.
\end{proof}
\end{proof}

\section{Proof of Theorem \ref{hirzthem}}\label{proof1}
In this section we provide a proof of Theorem \ref{hirzthem} whose strategy is
different from that of Auroux \cite{auroux}.
\begin{prop}\label{hireqchokuseki}
For $0 < \alpha < 1$, $(S^2, (1-\alpha)\omega_{\rm std})  \times (S^2,
(1+\alpha) \omega_{\rm std})$
is symplectomorphic to $F_2(\alpha)$.
For $\alpha = 0$, $(S^2,\omega_{\rm std})  \times (S^2,\omega_{\rm std})$
is symplectomorphic to $\widehat F_2(0)$.
\end{prop}
From now on, we write
$$S^2(1+\alpha) \times S^2(1-\alpha)=(S^2, (1+ \alpha) \omega_{\rm std})  \times (S^2, (1-\alpha)\omega_{\rm std}).$$
Although this proposition is well-known, we give a proof for the later purpose
of computing the full potential of the resolution $\widehat F_2(0)$ of $F_2(0)$.
(See \cite{Mc}, \cite{MK}.)
\begin{proof}
We define a holomorphic vector bundle $\mathcal{V} \to \C \times \C P^1$ as follows.
Write $U_0= \C P^1 \setminus \{\infty\}$ and $U_{\infty}=\C P^1 \setminus \{0\}$ and take
the transition function
\begin{equation}\label{LO-transit}
f(a,z):(v_1, v_2) \in \C ^2 \mapsto (zv_1+av_2,z^{-1}v_2) \in \C^2,
\end{equation}
for $(a,z) \in \C \times (U_0\cap U_{\infty})$.
The restriction $\mathcal{V}_a := \mathcal{V}|_{\{a\} \times \C P^1}$ is isomorphic to
the bundle $\mathcal{O}(-1) \oplus \mathcal{O}(1)$ over $\C P^1$ for $a = 0$ while it is holomorphically trivial for $a \neq 0$.

Take its projectivization
$$
\pi: \mathcal{X}=\P (\mathcal{V})\to \C  \times \C P^1.
$$
Let $p_i$ be the projection from $\C \times \C P^1$ to the $i$-th factor.
We denote
$$
X_a = \mathcal X|_{\{a\} \times \C P^1}
$$
as a complex manifold and denote its complex structure by $J^a$.
\par
By the construction and the definition of Hirzebruch surfaces (see e.g., \cite{MK}), we have
\begin{lem}\label{XaX0}
$X_a$ is biholomorphic to $\C P^1 \times \C P^1$
for $a \neq 0$ but $X_0$ is biholomorphic to $F_2$.
\end{lem}

Recall that any two cohomologous K\"ahler forms on a compact
manifold with a fixed complex structure are isotopic by Moser's
theorem. We will define K\"ahler forms on $X_a$ in suitable
cohomology classes. We would like to note that our argument is
basically gluing $F_0(0) \setminus \{ O \}$ with a local model of
simultaneous resolution of $\C^2/\{\pm 1\}$ and does not require the
following specific construction.

Note that $X_0$ contains a $(-2)$-curve $C_{-2}$ and its normal bundle in ${\mathcal X}$
is ${\mathcal O}(-1) \oplus {\mathcal O}(-1)$ \cite{L-O}.
Therefore we have a contraction of  $C_{-2}$.
Here we give such a map explicitly.
Firstly we define a holomorphic map from $\mathcal X$ to $\C P^4$ as follows.
By the definition of $\mathcal V$, we have
$$\P ({\mathcal V}) = \C \times U_0 \times (\C \oplus \C) \ \bigcup \
\C \times U_{\infty} \times (\C \oplus \C ),$$
where $(a,[1,z],[v_1,v_2]) \in \C \times U_0 \times (\C \oplus \C)$ and
$(a,[z,1],(z v_1 + a v_2, z^{-1} v_2) \in \C \times U_{\infty} \times (\C \oplus \C)$
are identified.
Define
$
\Phi_0: {\mathbb C} \times U_0 \times {\mathbb P}^1 \to {\mathbb P}^4
$
by
$$
(a,[1,z_1],[v_1,v_2]) \mapsto [a v_2 +z_1 v_1,v_1,z_1 a v_2 +z_1^2 v_1, z_1 v_1.v_2]
$$
and
$
\Phi_{\infty}:{\mathbb C} \times U_{\infty}  \times {\mathbb P}^1 \to {\mathbb P}^4
$
by
$$
(a,[z_0,1],[u_1,u_2]) \mapsto [z_0 u_1, z_0^2 u_1 - a z_0 u_2, u_1, z_0 u_1 - a u_2, u_2].
$$
Then we find that $\Phi_0$ and $\Phi_{\infty}$ are glued to a map
$\Phi: {\mathbb P}({\mathcal V}) \to {\mathbb P}^4$.

Let $[\xi_1,\xi_2,\xi_3,\xi_4,\xi_5] $ be the homogenous coordinates
on ${\mathbb P}^4$. Then  the image of $\Phi$ is described by the
equation $\xi_1 \xi_4 = \xi_2 \xi_3$. Note that $\xi_1 - \xi_4 = a
\xi_5$. By changing the coordinates
$$\eta_1=\xi_1 + \xi_4, \quad  \eta_2 =\xi_2 - \xi_3, \quad  \eta_3= i(\xi_2 + \xi_3), \quad
\eta_4=i(\xi_1 - \xi_4), \quad  \eta_5 = \xi_5,$$
the equation becomes $\eta_1^2 + \eta_2^2 + \eta_3^2 + \eta_4^2 = 0$.
From the relation $\eta_4 = i a \eta_5$, we find that the image of
$X_a \subset {\mathbb P}({\mathcal V})$ is described by
$$ \eta_1^2 + \eta_2^2 + \eta_3^2  = a^2 \eta_5^2, \quad \eta_4= i a \eta_5.$$

Note that $\Phi|_{X_0}$ becomes $(\C^*)^2$-equivariant with respect
to a suitable $(\C^*)^2$-action on $\C P^4$. For a later
convenience, we modify the Fubini-Study form on $\C P ^4$ to a
K\"ahler form $\Omega_{\C P^4}$ so that the corresponding K\"ahler
metric is flat in a small neighborhood of $\Phi (O)$. We may also
assume that $\Omega_{\C P^4}$ is invariant under the action of the
maximal torus $S^1 \times S^1 \subset (\C^*)^2$. Then we define
$$\Omega= \Phi^* \Omega_{\C P^4}.$$

For $a=0$, the restriction of $\Omega$ is K\"ahler away from the
$(-2)$-section $C_{-2}$, along which $\Omega$ degenerates, and
semi-positive with respect to the complex structure on $F_2$.

For $\epsilon > 0$, we set
$$
\Omega_{\epsilon}=\Omega + \epsilon (p_2 \circ \pi)^*\omega_{\C P^1},
$$
where $\omega_{\C P^1}$ is the Fubini-Study form on $\C P^1$. Then
the restriction of $\Omega_{\epsilon}$ to any $X_a$ is K\"ahler, and
hence it follows from Moser's theorem that for each given $\e > 0$
$(X_a,\Omega_{\epsilon}|_{X_a})$ are symplectomorphic for all $a \in
\C P^1$.

Set $\epsilon = 2 \alpha/(1 - \alpha)$ and take a suitable scale change,
\begin{equation}\label{Omegaalpha}
\Omega^{\alpha}=(1-\alpha)\Omega_{\frac{2 \alpha}{1 - \alpha}}.
\end{equation}
We define a symplectic form on $X_a$
\begin{equation}\label{omegaalphaa}
\omega^\alpha_a = \Omega^\alpha|_{X_a}.
\end{equation}
Then it follows that $(X_0, \omega^{\alpha}_0)$ is symplectomorphic to $(X_a, \omega^{\alpha}_a)$
by the above discussion. On the other hand, the above discussion together
with Lemma \ref{XaX0} implies that $(X_0, \omega^{\alpha}_0)$
is symplectomorphic to $F_2(\alpha)$ and $(X_a, \omega^{\alpha}_a)$ is
symplectomorphic to
$S^2(1-\alpha) \times S^2(1+\alpha)$ for $a\neq 0$.
Combination of these now proves the statement of the proposition for $0 < \alpha < 1$.

Now we consider the case $\alpha = 0$.
We write $\widetilde{p}:= p_1 \circ \pi: \mathcal{X} \to \C$.
Denote by $\widetilde{\mathcal{X}}$ the image of $\mathcal{X}$ by
$ \widetilde{p} \times \Phi$.
The projection $\widetilde{p}$ descends to $p:\widetilde{\mathcal{X}} \to \C$.
The $(-2)$-section $C_{-2}$ of $X_0$ is contracted to the $A_1$-singularity by
$\Phi$.  Other $X_a$ are mapped to their images
biholomorphically.  Therefore, $p:\widetilde{\mathcal{X}} \to \C$ gives a smoothing
of the $A_1$-singularity.  Hence $X_a$ is isomorphic to $\widehat{F}_2(0)$.
For $a \neq 0$, the restriction $\Omega^{\alpha=0}|_{X_a} = \omega^0_a$ is a K\"ahler form,
which represents a cohomology class proportional to
$$
PD[\C P ^1 \times \{ pt \}] + PD [\{ pt \} \times \C P^1].
$$
Hence we obtain the statement for $\alpha = 0$.
\end{proof}
Using Proposition \ref{hireqchokuseki}
we describe our family of Lagrangian submanifolds as follows.
\par
Equip $\mathcal{X}$ with $\Omega^{\alpha}$, $\alpha > 0$.  Then the central fiber
$X_0$ is $F_2(\alpha)$ as a K\"ahler manifold.  We now describe its toric structure.
Consider the K\"ahler $S^1$-action on $\C P^1$.  Denote by  $\{ D_s \}$ the domain
bounded by an $S^1$-orbit $C_s$ such that the orientation as the boundary of $D_s$
is the same as the direction by the $S^1$-action and the area of $D_s$ is $s/2$ of the
area of $\C P^1$.
In particular,  $C_1$ bounds two discs in $\C P^1$ of equal area.
We lift this action holomorphically to
$$
\mathcal{V}_0 \cong \mathcal{O}(-1) \oplus \mathcal{O}(1) \to \{0\} \times \C P^1.
$$
Such a lift is not unique, but unique up to the $S^1$-action obtained by the multiplication
of $e^{ik\theta}$ for some integer $k$.
The lifted $S^1$-action and the action of $S^1 \times S^1$ by fiberwise multiplication
commute and give a holomorphic $T^3$-action, which induces
the $T^2$-action on $X_0$.  Here $T^2$ is the quotient of $T^3$ by the diagonal
subgroup in the second and third factors.
This $T^2$-action on $(X_0, J^0)$ is isomorphic to the $T^2$-action on $F_2(\alpha)$
with the standard complex structure as a toric manifold.
In particular, $J^0$ is isomorphic to $T^2$-invariant complex structure on the toric
Hirzebruch surface $F_2$.
In this description,  the torus fibers $L(u_1,u_2) \subset F_2(\alpha)$ for
$(u_1,u_2) \in \operatorname{Int} P(\alpha )$ are identified with the $T^2$-orbits.
We find that
$L(u_1,u_2) \subset (p_2 \circ \pi|_{X_0})^{-1}{C_{u_1/(1-u_2)}} \subset X_0$.
In particular, when $u_1 + u_2 =1$, it is contained in $(p_2 \circ \pi|_{X_0})^{-1}{C_1}$.
\par
Next, we consider a hypersurface $\SS=\pi^{-1}(\R \times \C P^1)$
and the characteristic foliation
$\operatorname{Char}_{\Omega^{\alpha}}(\SS)$.
\begin{defn}\label{psietc}
For $a \in \R$ and $\alpha > 0$,
the map $\psi_a^{\alpha}: X_0 \to X_a$ is
the symplectomorphism induced by integration of the characteristic
foliation $\operatorname{Char}_{\Omega^{\alpha}}(\SS)$.
When $\alpha =0$, we obtain an embedding $\psi_a^0 :X_0 \setminus C_{-2} \to
X_a$ in a similar way. We define the tori
\begin{equation}\label{Laalpha}
L_a^{\alpha}(u_1,u_2)=\psi_a^{\alpha} (L(u_1,u_2)) \subset (X_a,\omega^{\alpha}_a)
\end{equation}
which are Lagrangian with respect to $\omega_a^{\alpha}$.
In particular, for $a \neq 0$, we obtain Lagrangian tori $L_a^0(u,1-u)=\psi_a^0(L(u,1-u))$ in
$(X_a, \omega_a^0)$.
\end{defn}

Here we explain how the construction above is related to the one of $\widehat{F}_2(0)$
given in section \ref{Deformation}.
Let $N$ be a tubular neighborhood of $C_{-2}$ in $X_0$ such that the characteristic
foliation on $\partial N$ is isomorphic to the fibers of $S^2/\pm 1 \to \C P^1$.
Denote by $W$ a tubular neighborhood of the vanishing cycle in $X_a$, which is
bounded by $\psi_a^0(\partial N)$.
Then we have $X_a= \psi_a^0(X_0 \setminus N) \cup W$.
Applying the symplectic cutting construction along the boundary of $W$, we obtain
a closed symplectic $4$-manifold $N$ containing $+2$-curve $S$ and the symplectic form is
exact on the complement of $S$.
Then the work of McDuff \cite{MD} implies that $N$ is symplectomorphic to the product of
two copies of $S^2$ of  the same area.
It implies that $W$ is symplectomorphic to a disc bundle of $T^* \bf{S} ^2$.
Therefore we find that $(X_a, \omega_a^0)$, $a \neq 0$, is symplectomorphic to
$\widehat{F}_2(0)$.

Pick a symplectomorphism $\varphi_a: (X_a, \omega_a^0) \to S^2(1) \times S^2(1)$.
Then the Lagrangian tori  $T(u)$ in Theorem \ref{maintheorem} are given by
\begin{equation}\label{Lu}
T(u)= \varphi_a (L_a^0(u,1-u)).
\end{equation}

We use the next result for the proof of Theorem  \ref{hirzthem}.
\begin{theorem}\label{critcoincides}
The critical values of the potential function of $S^2(1-\alpha) \times S^2(1+\alpha)$
are equal to those of $F_2(\alpha)$. In fact, they are equal to the
eigenvalues of the quantum multiplication of the first Chern class.
\end{theorem}
\begin{proof}
The fact that the critical values of the potential function
are equal to the eigenvalues of the quantum multiplication of the first Chern class
is proved for the Fano toric manifolds in \cite{toric1} (See Remark 5.3 and Theorem 1.9 therein.
We note $S^2(1-\alpha) \times S^2(1+\alpha)$ is Fano.)
\par
Since $F_2(\alpha)$ is not Fano, we need to prove the corresponding fact separately.
We prove this by combining the following facts:
There exists an isomorphism $\varphi$ from the small quantum cohomology ring of $F_2(\alpha)$
to the Jacobian ring $\text{\rm Jac}(\frak{PO})$ of the potential function
$\frak{PO}$ of $F_2(\alpha)$.
Moreover $\varphi$ sends $c_1(F_2(\alpha))$ to the element $\frak{PO} \in \text{\rm Jac}(\frak{PO})$.
We can prove it by the argument of \cite{toric1} Remark 6.15.
(See \cite{toric3} for detail.)
\par
We have thus proved the second half of Theorem \ref{critcoincides}.
The first half follows from the second half and Proposition \ref{hireqchokuseki}.
\end{proof}
\par
A key idea of the proof of Theorem \ref{hirzthem} is our usage of Theorem \ref{critcoincides}
in our computation of the coefficient $c$.
Both $F_2(\alpha )$ and $S^2(1 -\alpha ) \times S^2(1 + \alpha )$ are
toric manifolds.   Hence the potential functions are defined for them.
Although $F_2(\alpha )$ and $S^2(1 -\alpha ) \times S^2(1 + \alpha )$ are not isomorphic
as toric K\"ahler manifolds, they are symplectomorphic. Therefore their quantum cohomologies
are isomorphic.  In particular, the eigenvalues of the quantum multiplication by the first
Chern class are the same.
By comparing the critical values of potential functions in the two pictures,
we will determine the coefficient $c$. The detail is now in order.
\par
The moment polytope of $S^2(1-\alpha) \times S^2(1+\alpha)$ is
$$
P'(\alpha) = \{(u_1,u_2) \mid 0 \le u_1 \le 1-\alpha, \,\,0 \le u_2 \le 1+\alpha \}.
$$
There is a unique balanced fiber over ${\mathbf u} = ((1-\alpha)/2,(1+\alpha)/2)$, where
the potential function is
$$
T^{(1-\alpha)/2}(y_1+y_1^{-1}) +  T^{(1+\alpha)/2}(y_2+y_2^{-1}).
$$
It has 4 critical points, $(y_1,y_2) = (\pm 1,\pm 1)$. Their associated critical values are
\begin{equation}\label{crival1}
\pm 2T^{(1-\alpha)/2}(1 \pm T^{\alpha}).
\end{equation}
\par\smallskip
On the other hand the balanced fiber of  $F_2(\alpha)$ is located at
$((1+\alpha)/2,(1-\alpha)/2)$, where the corresponding potential function is
\begin{equation}
\frak{PO}^u
=
T^{(1-\alpha)/2}(y_2 + (1+c)y_2^{-1}) +
T^{(1+\alpha)/2}(y_1 +y_1^{-1}y_2^{-2}).
\end{equation}
The condition for $(y_1,y_2)$ being critical is
\begin{eqnarray}
0 &=& 1 - y_1^{-2}y_2^{-2}.  \label{crit1}\\
0 &=& 1 -2 T^{\alpha}y_1^{-1}y_2^{-3} - (1+c)y_2^{-2}.\label{crit2}
\end{eqnarray}
The equation (\ref{crit1}) implies $y_1y_2 = \pm 1$.
\par\smallskip
\emph{Case 1}:  $y_1y_2 = -1$.  (\ref{crit2}) implies $y_2^2 = 1 + c - 2T^{\alpha}$.
Then the critical values are
\begin{equation}\label{crival2}
\pm 2 T^{(1-\alpha)/2} \sqrt{1+c - 2T^{\alpha}}.
\end{equation}
\par\smallskip
\emph{Case 2}:   $y_1y_2 = 1$.  (\ref{crit2}) implies $y_2^2 = 1 + c + 2T^{\alpha}$.
Then the critical values are
\begin{equation}\label{crival3}
\pm 2 T^{(1-\alpha)/2} \sqrt{1+c +2T^{\alpha}}.
\end{equation}
We have already found  four different solutions, while the number (counted with multiplicity) of critical points
is the sum of Betti numbers, which is four in this case.
Therefore we must have
$$
1 \pm T^{\alpha} = \sqrt{1+c \pm2T^{\alpha}}
$$
from which $c=T^{2\alpha}$ follows immediately. This finishes the proof of Theorem \ref{hirzthem}.
\qed\par\smallskip

\section{Proof of Theorem \ref{POsingptt}}\label{proof2}
We first insert $c = T^{2\alpha}$ obtained in Theorem \ref{hirzthem} into
the formula (\ref{POord2}) and obtain the potential function
\begin{equation}\label{POF2alpha}
\frak{PO} =\frak{PO}(y_1,y_2;u_1, u_2) =
T^{u_1}y_1 + T^{u_2}y_2 +T^{2-u_1-2u_2}y_1^{-1}y_2^{-2} +
(1+T^{2\alpha})T^{1-\alpha -u_2}y_2^{-1}
\end{equation}
for $F_2(\alpha)$ for any $0 < \alpha < 1$. As we pointed out before,
Theorem \ref{POsingptt} is obtained by
formally setting $\alpha = 0$ in \eqref{POF2alpha}.
\par
In this section, we give two different justifications of
this `formal insertion' to obtain the potential function of $\widehat F_2(0)$.
Both proofs use Theorem \ref{hirzthem}.
The first proof is based on the deformation family
$\widetilde{p}=p_1 \circ \pi: \mathcal{X} \to \C$
constructed in the proof of Proposition \ref{hireqchokuseki}.
The second proof uses the standard gluing argument. The latter may be applied to more general cases,
while the former uses a special feature of Hirzebruch surfaces.
\par
We remark that potential function $\frak{PO}$ of Lagrangian submanifold
as a $\Lambda_0$-valued function on $H^1(L;\Lambda_{0})$
is well-defined up to a change of coordinate
that is congruent to identity. (See Remark \ref{welldefrem}.)
In the case of toric fiber, we can use a $T^n$-equivariant
perturbation so that $\frak{PO}$ is strictly well-defined as
a function on $H^1(L;\Lambda_{0})$.
We emphasize that we are studying the case of Lagrangian submanifolds which is \emph{not}
of toric fiber.
\par
Precisely speaking, the first proof implies the following:
For each
$(u_1,u_2)$, there exists a tame almost complex structure $J_{u_1,u_2}$ on
$\widehat F_2(0)$ such that we have strict equality
(\ref{POsingpt}) for the potential function with respect to this $J_{u_1,u_2}$.
\par
The second proof implies the following:
For each $(u_1,u_2)$ and $E$ there exists a compatible
almost complex structure $J_E$ on $\widehat F_2(0)$ such that
\begin{equation}\label{POF2alphaTE}
\frak{PO}^{J_E}(y_1,y_2;u_1, u_2) \equiv
T^{u_1}y_1 + T^{u_2}y_2 +T^{2-u_1-2u_2}y_1^{-1}y_2^{-2} +
2T^{1 -u_2}y_2^{-1}
\mod T^E.
\end{equation}
holds for the potential function $\frak{PO}^{J_E}$ of $L(u_1,u_2)$
with respect to the almost complex structure $J_E$.
\par
Both the first and the second proofs imply the equality
(\ref{POsingpt}) modulo a coordinate change of $(y_1,y_2)$ congruent
to the identity modulo $\Lambda_+$. To prove Theorem
\ref{maintheorem}, this statement is sufficient.
\subsection{Proof I: Deformation method}

Let $\operatorname{Int}P = \bigcup_{\alpha>0}\operatorname{Int}P(\alpha)
= \{
(u_1,u_2) \in \R^2 \mid u_i > 0,
u_2 < 1, u_1+2u_2 < 2
\}$.
We put $\mathcal{X}_{\alpha}= (\mathcal{X},\Omega_{\alpha})$,
where $\alpha \in [0,1)$ and  $\mathcal{X} = \mathbb P(\mathcal V)$  is as
in section \ref{proof1}.
As a family of $C^{\infty}$-manifolds, we have a trivialization $\bigcup_{a \in \C} X_a \cong
\C \times (S^2 \times S^2)$.
When we specify the symplectic form $\omega_a^{\alpha}$ on $X_a$,
we write $X_a^{\alpha}=(X_a,\omega_a^{\alpha})$.
Since $\Phi|_{X_0}$ is equivariant under the torus action, we may assume that
$\omega_0^{\alpha}$, $\alpha > 0$ is a toric K\"ahler form on $X_0$.

The complement of the (+2)-section $C_{2}$ in $p:X_0 \to \C P^1$ is
biholomorphic to the total space of $H^{\otimes -2}$, where $H$ is
the hyperplane section bundle of $\C P^1$. We equip $H^{\otimes -2}$
with a hermitian metric such that the corresponding hermitian
connection satisfies the following condition. Denote by $\omega_{\C
P^1}$ the Fubini-Study form on $\C P^1$ with total area 1 and by
$\theta$ the hermitian connection form of $H^{\otimes -2}$. Then the
curvature of the connection $\theta$ is  a multiple of $\omega_{\C
P^1}$.

Recall that  we modified the Fubini-Study form on $\C P^4$ such that
it becomes flat in a neighborhood of $\Phi (O)$ and invariant under
$S^1 \times S^1$-action. Then we can express
$$\omega_0^{\alpha}=(1-\alpha) \left( \frac12 d(\rho (r)^2 \theta) + \frac{2\alpha}{1-\alpha} p^*\omega_{\C P^1} \right),$$
where $r$ is the fiberwise norm and $\rho:[0,\infty) \to [0, 1)$ is a strictly
increasing smooth function with $\rho (r)=r$ around $r=0$.
We will modify the function $\rho$ according to $(u_1,u_2)$ later.

Recall from \eqref{Omegaalpha}, \eqref{omegaalphaa} that we have
$$
\omega_0^{\alpha}=(1-\alpha) \left(\omega_0^0 +
\frac{2\alpha}{1-\alpha} (p_2 \circ \pi)^* \omega_{\C P^1} \right).
$$
Therefore we obtain
$$\omega_0^0 = \frac12 d(\rho (r)^2 \theta).$$
When we collapse the $(-2)$-curve $C_{-2}$ in $X_0$, we obtain $F_2(0)$
with an isolated singular point $O$.
Note that $X_0 \setminus (C_2 \cup C_{-2}) \cong F_2(0) \setminus (C_2 \cup \{O\})$
is biholomorphic to $(-\infty,\infty) \times S^3/\{\pm 1\}$ with a complex structure, which
is invariant under the translation in $(-\infty,\infty)$.
In the cylindrical coordinates,  we have
$$
\omega_0^0 = \frac{1}{2} d(e^{2 \sigma(s)} \theta),
$$
where $s$ is the coordinate on $(-\infty, \infty)$
and $\sigma(s)=\log \rho (e^s)$.   In particular,
$\sigma(s)=s$ for sufficiently small $s$.

\begin{lem}\label{diffeo}
For $(u_1,u_2) \in \operatorname{Int}P $, we consider the Lagrangian submanifold $L_a^{\alpha}(u_1, u_2)$
given in (\ref{Laalpha}). Then there is a constant $\delta(u_1,u_2)>0$ with the following properties:
\begin{enumerate}
\item
If $|a| < \delta(u_1,u_2)$, $\alpha < \delta(u_1,u_2)$, we have a diffeomorphism
$\phi_{a,(u_1,u_2),\alpha}: X_a \to X_a$ such that
$$L_a^{\alpha}(u_1, u_2)=\phi_{a,(u_1,u_2),\alpha}(L_a^{\alpha} (1/2,1/2)).$$
\item
There is a neighborhood $U(u_1,u_2)$
of $C_{-2}$ in $\bigcup_{a} X_{a}$ such that
$\phi_{a,(u_1,u_2),\alpha}$ is the identity map on $U(u_1,u_2) \cap X_a$.
Here we regard $C_{-2} \subset X_0 \subset \mathcal{X}_{0}$.
\item
There exists a smooth family of almost complex structures $J^{(1)}_{a;u_1,u_2}$
on $X_{a}$, which
is tamed both by $\omega^{\alpha}_a$ and $\phi_{a,(u_1,u_2),\alpha}^*\omega^{\alpha}_a$
for $|a|, \alpha < \delta(u_1,u_2)$.

\end{enumerate}
\end{lem}
\begin{proof}
Firstly, we consider the case that $a=\alpha=0$.
Since  the $(\C^*)^2$-action induced by the toric
structure is transitive on the set of all torus fibers $L(u_1,u_2)$, we have a biholomorphic map $g_{u_1,u_2}$ such that
$g_{u_1,u_2}(L(1/2,1/2))=L(u_1,u_2)=L_0^{0}(u_1,u_2)$.
We modify $g_{u_1,u_2}$ to obtain $\phi_{0,(u_1,u_2),0}$ and find
$J^{(1)}_{0;u_1,u_2}$ so that items 1-3 of Lemma \ref{diffeo}
are also satisfied. Namely we have:
\begin{sublem}\label{mainsublem}
There exist a neighborhood $U \subset X_0$ of $C_{-2}$ depending only on $(u_1,u_2)$,
and a map $\phi_{0,(u_1,u_2),0} : X_0 \to X_0$,
and an almost complex structure $J^{(1)}_{0;u_1,u_2}$
such that the following holds:
\begin{enumerate}
\item
$L_0^{0}(u_1, u_2)=\phi_{0,(u_1,u_2),0}(L_0^{0} (1/2,1/2))$.
\item
$\phi_{0,(u_1,u_2),0}$ is an identity map on $U$.
\item
$J^{(1)}_{0;u_1,u_2}$
is tamed both by $\omega^{\alpha}_0$ and $\phi_{0,(u_1,u_2),0}^*\omega^{\alpha}_0$.
\end{enumerate}
\end{sublem}
\begin{proof}
We identify $(\C^2 \setminus \{0\})/\{\pm 1\}$ with $(-\infty, \infty) \times S^3/\{\pm 1\}$,
which is $(\C^*)^2$-equivariantly biholomorphic to $X_0 \setminus (C_2 \cup C_{-2})$.
From now on, we fix this identification.
Here $O \in \C^2/\{\pm 1\}$ corresponds to
the limit of $\{s\} \times S^3/\{\pm 1\}$ as $s$ tends to $-\infty$.
The complex torus $(\C^*)^2$ acts on $\C^2/\{\pm 1\}$ by
$$
(w_1,w_2) \cdot [z_1,z_2] = [w_1z_1,w_2z_2],
$$
where $(w_1,w_2) \in (\C^*)^2$ and $(z_1,z_2) \in \C ^2$.
Note that this map extends to a biholomorphic automorphism of $X_0$.
(The action by $(\C^*)^2/\{\pm 1\}$ is the torus action as a toric manifold.)

\par
For $(u_1,u_2)$, we pick $a < b \in \R$ so that
$$
L(u_1,u_2), \ L(1/2,1/2) \subset (a,b) \times S^3/\{\pm 1\} \subset X_0.
$$
We choose the function $\rho$ such that $\rho(s)=s$ on $(-\infty,b]$.
We set
$$
i(c_1,c_2) = (c_1c_2, c_1^{-1}c_2).
$$
Since $(\R_+)^2 \subset (\C^*)^2$ acts on the set of torus fibers $L(u_1,u_2)$ transitively,
there is $(c_1,c_2) \in (\R_+)^2$ such that

$$
i(c_1,c_2) \cdot L(1/2,1/2) = L(u_1,u_2).
$$
Thus we may identify  $g_{u_1,u_2}$ with the multiplication by
$i(c_1,c_2) \in (\C^*)^2$.
\par
Let $S$ be a sufficiently large positive number
to be chosen later.
We will modify the action by $i(c_1,1)$ in the region $(-\infty,-S) \times S^3/\{\pm 1\}$ so that it
becomes the identity on a neighborhood of $C_{-2}$.

Identifying $S^3/\{\pm 1\} = (\C^2 \setminus \{0\})/\R^*$, we denote
by $[[x_1,x_2]] \in S^3/\{\pm 1\}$ the equivalence class of $(x_1,x_2) \in \C^2 \setminus \{0\}$.
We define $\phi_1 : (-\infty, \infty) \times S^3/\{\pm 1\} \to (-\infty, \infty) \times S^3/\{\pm 1\}$
by
$$
\phi_1 (s,[[x_1,x_2]]))
= (s + h_S(s,[[x_1,x_2]]),[[c_1^{f_S(s)}x_1,c_1^{-f_S(s)}x_2]])
$$
where $f_S$ is a nondecreasing function such that
$$
f_S(s) =
\begin{cases}
0  & s < -3S+1, \\
1 & s > -2S
\end{cases}
$$
and
$$
h_S(s,[[x_1,x_2]])=\frac12 \{ \log (|c_1^{f_S(s)} x_1|^2+ |c_1^{-f_S(s)} x_2|^2)
- \log (|x_1| ^2+ |x_2|^2) \}.
$$
Note that $h_S$ is a bounded function.
We next extend $\phi_1$ to a map $X_0 \to X_0$ as follows.
$X_0 \setminus \{ (-3S,-S) \times S^3/\{\pm 1\} \}$ has two connected components $V_1$, $V_2$ such that $C_{-2} \subset V_1$.
We define
$$
\phi_1 = \begin{cases} id \quad & \text{on $V_1$}\\
i(c_1,1) \quad & \text{on $V_2$}.
\end{cases}
$$
Then if we take $f_S$ so that $| df_S/ds | < 2/S$, it is easy to see that for sufficiently large $S$
(depending only on $c_1,c_2$), $\phi_1$ is a diffeomorphism of $X_0$ and $J_0$ is tamed by $\phi_1^*\omega_0$.
We also choose $S$ large enough such that $- S <  a - \max |h_S| - 1$.
\par
We next define $(-\infty, \infty) \times S^3/\{\pm 1\} \to (-\infty, \infty) \times S^3/\{\pm 1\}$ by
$$
\phi_2(s,[[x_1,x_2]]) = (\chi_S(s), [[x_1,x_2]]).
$$
Here $\chi_S:(-\infty,\infty) \to (-\infty, \infty)$ is a strictly increasing smooth function
such that
$\chi_S(s)=s$ for $s < a-S$ and $\chi_S(s)=s+ \log c_2$ for $s \geq a$.
Using the fact that $d\chi_S/ds > 0$, we find that $J_0$ is tamed by $\phi_2^{*}\omega_0$.
Note also that $\phi_2$ coincides with the action by $i(1,c_2)$ in the region
$[a, \infty) \times /\{ \pm 1\}$.  In particular, $\phi_2$ extends to $C_2$ smoothly.
\par
We remark that $\phi_2$ is the identity map on the image of $(-\infty,-S) \times S^3/\{\pm 1\}$ by
$\phi_1$ and that $\phi_1$ is holomorphic on $(-2S, \infty) \times S^3/\{\pm 1\}$.
Therefore $J_0$ is tamed by $(\phi_2\phi_1)^{*}\omega_0$.
\par
Since $\phi=\phi_2\phi_1$ is the action of $(c_1,c_2)$ on $X_0 \setminus (C_{-2} \cup
(-\infty, -2S) \times S^3/\{\pm 1\}$, it follows that
$J^{(1)}_{0;u_1,u_2}$ is tamed by $\phi^{*}\omega_0$.
\par
By definition $\phi(L(1/2,1/2)) = L(u_1,u_2)$.
Since $\phi$ is the identity map near $C_{-2}$ and biholomorphic near $C_2$,  $\phi$, resp. $\phi_*J$  extend to a diffeomorphism  $\phi_{0,(u_1,u_2),0}$
of $X$,
resp. a complex structure $J_{0;u_1,u_2}$, which have the required properties.
\end{proof}
We next consider the case $(a,\alpha) \neq (0,0)$.
We set
\begin{equation}\label{phidef}
\phi_{a,(u_1,u_2),\alpha}=\psi_a^{\alpha} \circ \phi_{0,(u_1,u_2),\alpha} \circ (\psi_a^{\alpha})^{-1}.
\end{equation}
(See Definition \ref{psietc}.)
Note that $\phi_{a,(u_1,u_2),\alpha}$ is the identity map on $\psi_a^{\alpha}(U)$, which is a neighborhood of the vanishing cycle.
Therefore $\phi_{a,(u_1,u_2),\alpha}$ is defined as a diffeomorphism on $X_a$.
We have $L_a^{\alpha}(u_1, u_2)= \phi_{a,(u_1,u_2),\alpha}(L_a^0 (1/2,1/2))$.
\par
Fix $(u_1,u_2) \in \operatorname{Int}P$.
We extend $J_{0;u_1,u_2}^{(1)}$ to a smooth family
of almost complex structures  $J_{a;u_1,u_2}^{(1)}$ on $X_{a}$ such that
$J^{(1)}_{a;u_1,u_2}=J^a$ on $\psi_a^{\alpha}(U)$ for $|a|, \alpha < \delta$.
Here we pick and fix a positive constant $\delta$.

\par
Pick an open subset $U'$ such that $C_{-2} \subset U' \subset \overline{U}' \subset U$.
The condition that $J_{a;u_1,u_2}^{(1)}$ is tamed both by $\omega^{\alpha}_a$ and
$(\psi_a^{\alpha} \circ \phi_{0,(u_1,u_2),\alpha} \circ (\psi_a^{0})^{-1})^* \omega^{\alpha}_a$
on $X_0 \setminus U'$ is an open condition for $a, \alpha$.  ($\omega_0^0$ degenerates
along $C_{-2}$.)
Note also that the above taming condition holds on $X_0 \setminus U$ at $a=\alpha =0$
and $J^a$ is compatible with $\omega_a^{\alpha}$ for all $\alpha \geq 0$ when $a \neq 0$.
Therefore we find $\delta(u_1,u_2)$ enjoying the required two properties.
\end{proof}

We identify $(X_0,\omega_a^\alpha)$ with the toric $F_2(\alpha)$ by a
fixed symplectomorphism.
Consider $L(1/2,1/2) = L_0^{\alpha}(1/2,1/2)
\subset F_2(\alpha) = X_0$.
For $a \in \R \setminus \{0\}$,
$$
L_a^0(1/2,1/2)=\psi_a^0 (L(1/2,1/2)) \subset X_{a,0}
$$
is a Lagrangian submanifold with respect to $\omega^{0}_a$.
In particular, it is a monotone Lagrangian submanifold with respect to
$\omega^0_a$.
The ambient space is also monotone.
Thus we find that the Maslov index of any non constant $J$-holomorphic disc with boundary on
the Lagrangian torus $L_a^0(1/2,1/2)$ is at least $2$.
Here $J$ is {\it any} almost complex structure tamed by $\omega_a^0$.
Namely Assumption \ref{Maslov 2} is satisfied for such an almost complex structure $J$.
\par
Using Theorem \ref{POMaslov 2} in Appendix
and the above remark, we find that the conclusion of Theorem \ref{hirzthem} holds for any such $J$
and $L(1/2,1/2)$.
(Theorem \ref{hirzthem} itself is the case of toric complex structure.)
In particular it holds for $J^{(1)}_{a;u_1,u_2}$.

\par
We put
$$
J^{(2)}_{a,\alpha;u_1,u_2} = (\phi_{a,(u_1,u_2),\alpha})_*J^{(1)}_{a;u_1,u_2}.
$$
$J^{(2)}_{a,\alpha;u_1,u_2}$ is tamed by $\omega_a^{\alpha}$ for sufficiently small $\alpha, a$.
\par
Clearly the moduli space of $J^{(2)}_{a,\alpha;u_1,u_2}$-holomorphic discs
with boundary on the torus $L_a^{\alpha}(u_1,u_2)$
is identified with the moduli space of $J^{(1)}_{a;u_1,u_2}$-holomorphic
discs which bound $L_a^{\alpha}(1/2,1/2)$. (Lemma \ref{diffeo}).
\par

Hence, for any $\beta$ with $\mu (\beta) = 2$ and $\alpha$, $a_1, \, a_2$ with $\alpha, |a_1|, |a_2| < \delta(u_1,u_2)$, we have
a cobordism between
$$
\MM(L_{a_1}^{\alpha}(u_1,u_2); J^{(2)}_{a_1,\alpha;u_1,u_2};\beta)
\cong \MM(L_0^{\alpha}(1/2,1/2);  J^{(1)}_{a_1;u_1,u_2};\beta)
$$
and
$$
\MM(L_{a_2}^{\alpha}(u_1,u_2),J^{(2)}_{a_2,\alpha;u_1,u_2};\beta)
\cong
\MM(L_0^{\alpha}(1/2,1/2);  J^{(1)}_{a_1;u_1,u_2};\beta).
$$

On the other hand, Gromov's compactness implies that when $\alpha >0$, for any given $E>0$
there is some $a_0 > 0$ such that
for $a \in (0,a_0)$, the potential functions of
$L_0^{\alpha}(u_1,u_2) \subset (X_0,\omega^{\alpha}_0,J^{(2)}_{0,\alpha;u_1,u_2})$ and
$L_a^{\alpha}(u_1,u_2) \subset (X_a, \omega^{\alpha}_a,J^{(2)}_{a,\alpha;u_1,u_2})$ coincide up to the order of $T^E$.
\par
Combining these facts, we conclude that the potential functions of
$L_0^{\alpha}(u_1,u_2) \subset (X_0,\omega^{\alpha}_0,J^{(2)}_{0,\alpha;u_1,u_2})$ and of
$L_a^{\alpha}(u_1,u_2) \subset (X_a, \omega^{\alpha}_a,J^{(2)}_{a,\alpha;u_1,u_2})$ coincide for $\alpha \le \delta(u_1,u_2)$
and {\it any} $a$ with $|a| < \delta(u_1,u_2)$.
Recall that Theorem \ref{hirzthem} gives the potential function of
$L_0^{\alpha}(u_1,u_2)$, hence we obtain the potential function of
$L_a^{\alpha}(u_1,u_2)$.
We put
$$
J_{u_1,u_2}^a = \lim_{\alpha \to 0}J^{(2)}_{a,\alpha;u_1,u_2}
= \lim_{\alpha \to 0}(\phi_{a,(u_1,u_2),\alpha})_*J^{(1)}_{a;u_1,u_2}.
$$
Since $\phi_{a,(u_1,u_2),\alpha}$ is the identity map on $\psi_a^{\alpha}(U)$,
$J^{(2)}_{a,\alpha;u_1,u_2}$ is independent of $\alpha$ around the vanishing cycle.
Therefore $J_{u_1,u_2}^a$ is well-defined.
Then we find that
$c_k=\deg [ev_0;\MM_1(X_a,L(u_1,u_2);J^a_{u_1,u_2};\beta_1+k\,[D_1]) \to L(u_1,u_2)].$
\par
The potential function of $L_a^{\alpha}(u_1,u_2)$ in $(X_a, \omega_a^{\alpha})$ depends
on $\alpha$ only through the exponents of $T$, i.e., $\omega_a^{\alpha}$-areas of
bordered stable maps.
Namely, $c_k$'s, which appear as coefficients in the potential function, does not depend
on $\alpha$.

When $a \neq 0$, we have a family of K\"ahler forms $\omega_a^{\alpha}$,
$\alpha \le \delta(u_1,u_2)$. (Namely the family $\omega_a^{\alpha}$ extends to $\alpha = 0$.)
Letting $\alpha \to 0$, we find that the potential function of
$L^{0}_a(u_1,u_2) \subset (X_a, \omega^0_a,J^a_{u_1,u_2})$ becomes
$$
\frak{PO}
= T^{u_1}y_1 + T^{u_2}y_2 +T^{2-u_1-2u_2}y_1^{-1}y_2^{-2} +
2T^{1 -u_2}y_2^{-1}.
$$
This is the conclusion of Theorem \ref{POsingptt}.
\qed

\subsection{Proof II: Gluing method}

Now we present the second proof.
Let $F_2(0)_0$ be the $F_2(0)$ minus the singular point.
We symplectically identify the end of $F_0(0)_0$ with the
semi-infinite cylinder
$$
((-\infty,0] \times S^3/\{\pm 1\},d(e^{s}\lambda))
$$
over the real projective space $\R P^3 = S^3/\{\pm 1\}$.
Here $\lambda$ is the standard contact form on $S^3/\{\pm 1\}$ and $S_0
\in \R$ is sufficiently large.
\par

\par
For each given ${\mathbf u} \in \text{\rm Int}P$, we consider the moduli space
\begin{equation*}
\aligned
\widetilde{\mathcal M}_{(1;1)}^\sharp(L({\mathbf u});F_2(0))  =  \{(v, z_-; z_0)) \mid \,
& z_-\in \text{Int } D^2, \,  z_0 \in \partial D^2, \\
&
v: (D^2 \setminus \{z_-\}, \partial D^2) \to (-\infty,0] \times S^1 \to F_2(0)_0  \\
& \mbox{ is a map with the following properties} \}:
\endaligned
\end{equation*}
\begin{enumerate}
\item $v$ is proper and pseudo-holomorphic.
\item There exists a loop $\gamma: S^1 \to S^3/\{\pm 1\}$ which is a simple
closed Reeb orbit and such that there exists $\tau_0 \in \R$ and $e^{it_0} \in S^1$ with
$$
\lim_{\tau\to-\infty} d(v(\tau,t), v_\gamma(\tau+\tau_0,t + t_0))  = 0
$$
where $z = e^{\tau + it}$ is an analytic coordinate such that $z=0$ at $z_-$,
$d$ is the cylindrical metric on $(-\infty, 0] \times S^3/\{\pm 1\}$
and $v_\gamma$ is a trivial cylinder in $(-\infty, 0] \times S^3/\{\pm 1\}$ defined by
$v_\gamma(\tau,t) = (\tau, \gamma(t))$.
\item $v(z) \in L({\mathbf u})$ for $z \in \partial D^2$.
\item We have
$$
\int v^*\omega \le 1-u_2+\delta
$$
where the constant $\delta$ is a fixed constant so small that the corresponding homotopy class
of such map $v$ is unique.
\end{enumerate}
We note that $PSL(2,\R)$ acts on
$\widetilde{\mathcal M}_{(1;1)}^\sharp(L({\mathbf u});F_2(0))$ by
$$
g\cdot (v, z_-, z_0) = (v\circ g^{-1}, g(z_-), g(z_0)).
$$
Then we define the moduli space $\mathcal M_{(1;1)}^\sharp(L({\mathbf u}); F_2(0))$
to be the quotient
$$
\mathcal M_{(1;1)}^\sharp (L({\mathbf u}); F_2(0))
= \widetilde{\mathcal M}_{(1;1)}^\sharp(L({\mathbf u}); F_2(0))/PSL(2,\R).
$$
\par\smallskip
We remark that the set $\widetilde{\mathcal R}_1(\lambda) : = \text{\rm Reeb}_1(S^3/\{\pm 1\})$
of (parameterized) Reeb orbits $\gamma$ given in (2) above
can be identified with the set of parameterized closed geodesic of $S^2$ with minimal length
and so it is diffeomorphic to $S^3/\{\pm 1\}$. Here the subscript `1' stands for the `minimal period'.
We take its quotient by $S^1$-action defined by changing the base point of the parametrization.
Let ${\mathcal R}_1(\lambda)$ be the corresponding quotient space, which is diffeomorphic to $S^2$.
\par
On the cylinder $\R \times S^3/\{ \pm 1\}$, we denote by $s$, resp. $\Theta$ the projection
to $\R$, resp. $S^3/\{ \pm 1\}$.  By abuse of notation, we also use $s$ and $\Theta$
on a subcylinder contained in $\R \times S^3/\{ \pm 1\}$.
There is an obvious asymptotic evaluation map
$$
ev^\sharp: \mathcal M_{(1;1)}^\sharp(L({\mathbf u}); F_2(0)) \to \mathcal R_1(\lambda) \cong S^2
$$
which assigns to $(v,z_-,z_0) \in \mathcal M_{(1;1)}^\sharp(L({\mathbf u}); F_2(0))$ its asymptotic Reeb orbit $\gamma$
defined by
$$
\gamma(t) = \lim_{\tau \to -\infty} \Theta \circ v(\tau/T, t/T)
$$
where $z = e^{\tau + it}$ is the analytic coordinates adapted to $z_-$
and $T$ is the period of the Reeb orbit $\gamma$. We note that the period $T$ is
determined by the asymptotic behavior of $v$ by the formula
$$
T = \lim_{\tau \to - \infty} \int (\Theta \circ v_\tau)^* \lambda
$$
where $v_\tau$ is the loop defined by $v_\tau(t) = v(\tau,t)$.
We have another evaluation map
\begin{equation}\label{evat1}
ev_0 : \mathcal M_{(1;1)}^\sharp (L({\mathbf u}); F_2(0)) \to L({\mathbf u})
\end{equation}
which assigns to $(v,z_-,z_0)$ the point $v(z_0) \in L({\mathbf u})$.
\par
It is by now well-known that there is an appropriate Fredholm theory
for the study of $\mathcal M_{(1,1)}^\sharp(L({\mathbf u}); F_2(0))$.
We omit explanation
of this Fredholm theory here just referring to
\cite{surgery} for detailed exposition given in a similar context.
See also \cite{hind}.
We can check that the moduli space
$\mathcal M_{(1,1)}^\sharp(L({\mathbf u}); F_2(0))$ has virtual dimension $2$
(See Remark \ref{dimrem}.) and minimality of energy implies
that it is compact and hence the mapping degree of (\ref{evat1}) is well-defined.
\par
We next resolve the singularity of $\C^2/\{\pm 1\}$ as we did for $F_2(0)$,
which gives rise to a $\C$-bundle over $\C P^1$ of degree $-2$.
Furthermore it carries a symplectic form that coincides with that of
$\C^2/\{\pm 1\}$ away from $[(0,0)]$ and is determined (up to symplectomorphism) by the area of
the zero section $\C P^1$. We denote by $2\pi(2\alpha)$ the symplectic area of $\C P^1$
and the resulting symplectic manifold by $X(\alpha)$. Furthermore $X(\alpha)$ carries
a natural $\C^*$-action induced by the one on $\C^2$.
\par
The end of $X(\alpha)$, $0 < \alpha < 1$, is symplectomorphic to the cylinder
$$
([0,\infty) \times S^3/\{\pm 1\},d(e^{s}\lambda)),
$$
where $s$ is the coordinate on $[0,\infty)$.
We consider the set of smooth maps
$
u : \C \to X(\alpha).
$
For the purpose of describing the gluing process precisely, we
identify $\C$ with $\C P^1 \setminus \{pt\}$, and so consider the
moduli space of maps
$$
w: \C  \to X(\alpha),
$$
with the following properties:
\begin{enumerate}
\item $w$ is proper and pseudo-holomorphic.
\item There exist $(\tau_0,t_0) \in \R \times S^1$ and $\gamma$ where $\gamma$ is
as in (2) in the definition of $\mathcal M_{(1,1)}^\sharp (L({\mathbf u}); F_2(0))$ such that
$$
\lim_{\tau\to+\infty}d(w(e^{\tau+\tau_0 + (t+t_0)i}), v_\gamma(\tau,t)) = 0.
$$
where $d$ is the cylindrical metric on the cylinder $([0,\infty) \times S^3/\{\pm 1\},d(e^{s}\lambda))$.
\end{enumerate}
We need to assume finiteness of an appropriate energy, more specifically the
Hofer energy \cite{hofer}, in addition. We omit the
precise formulation thereof because it is by now standard.
\par
We denote the asymptotic boundary of $X(\alpha)$ by $\partial_\infty X(\alpha)$ and
the relative (Moore) homology class
$$
[w] \in H_2(X(\alpha), \partial_\infty X(\alpha))
$$
of such $w$ is classified by its intersection number with the zero section $\C P^1 \subset X(\alpha)$.
Let $\tilde\beta$ be the class with intersection number $1$. Other classes
are $\tilde\beta + k\,[\C P^1]$. Let $\widetilde{\mathcal M}^\sharp_1(X(\alpha);k)$ be the moduli space of
such $w$ in homology class $[w] = \tilde\beta+ k\,[\C P^1]$.
\par
We take the quotient of the space $\widetilde{\mathcal M}^\sharp(X(\alpha);k)$ by
the $\text{\rm Aut}(\C)$-action  given by
$$
g\cdot w = w \circ g^{-1}
$$
and take its stable-map compactification. Denote the resulting compactified
moduli space by $\mathcal M^\sharp(X(\alpha);k)$.
There is an obvious asymptotic evaluation map
$$
ev^\sharp: \mathcal M^\sharp(X(\alpha);k) \to \mathcal R_1(\lambda)
$$
which assigns to $w$ an asymptotic Reeb orbit $\gamma$
given by
$$
\gamma(t) = \lim_{\tau \to \infty} \Theta \circ w(e^{(\tau+it)/T}),\quad
T = \lim_{\tau \to \infty} \int (\Theta \circ w_\tau)^*\lambda
$$
where $w_{\tau}(t) = w(e^{\tau+it})$.
\par\medskip
We next describe a family of almost complex structures we use.
\par
Let $\alpha$ and $S_0$ be positive numbers satisfying $e^{2S_0} \alpha < 1$.
We glue $F_2(0)_0 \setminus (-\infty,-S_0) \times S^3/\{\pm 1\}$
with $X(e^{2S_0}\alpha) \setminus [S_0,\infty) \times S^3/\{\pm 1\}$
along $\{-S_0\} \times S^3/\{\pm 1\}$ and
$\{S_0\} \times S^3/\{\pm 1\}$.
(We put the symplectic form $e^{-2S_0}d(e^{s}\lambda)$ on $X(e^{2S_0}\alpha)$.)
We then obtain $F_2(\alpha)$.
\par
This space contains the cylindrical region $[-S_0, S_0] \times S^3/\{\pm 1\}$.
Using this identification, we define a compatible
almost complex structure $J_{S_0}$, for each $S_0$,  on  $F_2(\alpha)$ so that
its restriction to $[-S_0, S_0] \times S^3/\{\pm 1\}$ is of product type and
its restriction to the other part is independent of $S_0$.
Then the zero section $\C P^1$ of $X(\alpha)$ becomes the $(-2)$-curve $D_1$ in
$F_2(\alpha)$.
\par
Now we consider the homology class of maps $(D^2,\partial D^2) \to
(F_2(\alpha),L({\mathbf u}))$.
Such a homology class is determined by the intersection numbers with irreducible components
of the toric divisor. We denote by the class $\beta_1$ the one that has intersection number 1 with
the $(-2)$-curve $D_1$ and 0 with all other irreducible components of
toric divisors. We consider the class $\beta= \beta_1 + k\,[D_1]$ and
let $\mathcal M_1(F_2(\alpha),L({\mathbf u});J_{S_0};\beta_1 + k\,[D_1])$ be the compactified moduli space of
stable maps  $(\Sigma,\partial \Sigma) \to (F_2(\alpha),L({\mathbf u}))$
of genus zero, in homology class $\beta_1 + k\,[D_1]$ and with one boundary marked point.
\par
We have the natural fiber product
$$
\mathcal M^\sharp(X(\alpha);k) \,\,{}_{ev^\sharp}\times_{ev^\sharp}
\mathcal M^\sharp_{(1,1)}(L({\mathbf u}); F_2(0))
$$
and the evaluation map
$$
ev_0 : \mathcal M^\sharp(X(\alpha);k) \,\,{}_{ev^\sharp}\times_{ev^\sharp} \mathcal M^\sharp_{(1,1)}(L({\mathbf u}); F_2(0)) \to L({\mathbf u})
$$
such that the following diagram
$$
\xymatrix{\mathcal M^\sharp(X(\alpha);k) \,\,{}_{ev^\sharp}\times_{ev^\sharp}
\mathcal M^\sharp_{(1,1)}(L({\mathbf u}); F_2(0))
\ar[rr]^(.6){\pi_2}\ar[dr]_{ev_0} & & \mathcal M^\sharp_{(1,1)}(L({\mathbf u}); F_2(0)) \ar[dl]^{ev_0} \\
&L({\mathbf u})&}
$$
commutes.
\begin{lem}\label{gluing}
For each $k,\, u$ and a constant $C> 0$ there exists $S_0(k,u,C)$ such that if $S_0 > S_0(k,u,C)$ and  $
C^{-1} \le e^{2S_0}\alpha\le C$ then
the virtual fundamental cycle of the evaluation map
$$
ev_0:\mathcal M^\sharp(X(e^{2S_0}\alpha);k) \,\,{}_{ev^\sharp}\times_{ev^\sharp} \mathcal M^\sharp_{(1,1)}(L({\mathbf u}); F_2(0)) \to L({\mathbf u})
$$
defines the same homology class in $L({\mathbf u})$ as that of
$$
ev_0 : \mathcal M_1(F_2(\alpha),L({\mathbf u});J_{S_0};\beta_1 + k\,[D_1]) \to L({\mathbf u}).
$$
\end{lem}
Using the fact that $\gamma$ is the Reeb orbit of smallest period, the lemma follows from a
standard gluing result. (See \cite{LR,surgery}.)

\begin{rem}\label{dimrem}
We can prove the equality
$$
\dim \mathcal M^\#(X(\alpha),k) = \dim \mathcal M^\#_{(1,1)}(L({\mathbf u});F_2(0)) = 2
$$
as follows. (Here $\dim$ is the virtual dimension that is the dimension as the
space with Kuranishi structure.)
\par
Since $c_1 =0$ in $X(\alpha)$, it follows that  $\dim \mathcal M^\#(X(\alpha),k)$ is
independent of $k$. In case $k=0$, our moduli space $\mathcal M^\#(X(\alpha),0)$
consists of the fibers of the $\C$ vector bundle $X(\alpha) \to \C P^1$.
We also find that the linearization operator is surjective.
Therefore $\dim \mathcal M^\#(X(\alpha),k) = 2$.
\par
On the other hand we have $\dim \mathcal M^\#_1(F_2(\alpha),L({\mathbf u});\beta_1)) = 2$
(see \cite{cho-oh} Theorem 5.1).
Therefore we obtain $\dim \mathcal M^\#_{(1,1)}(L({\mathbf u});F_2(0)) = 2$ from the
(analytic) index sum formula which is a part of the (standard) gluing result,
Lemma \ref{gluing}.
\end{rem}

\begin{lem}\label{degreeforv1}
For each $k,\, {\mathbf u}, \,  C$, there exists $S_0(k,{\mathbf u},C)$ such that if
$S_0 >S_0(k,{\mathbf u},C)$,
$C^{-1} \le e^{2S_0}\alpha\le C$,  and $ke^{2S_0}\alpha\le \delta_0({\mathbf u})$ then
the mapping degree of the map
$$
ev_0 : \mathcal M_1(F_2(\alpha),L({\mathbf u});J_{S_0};\beta_1+k\,[D_1]) \to L({\mathbf u})
$$
is equal to  $c_k$. Here $\delta_0({\mathbf u})$ is a sufficiently small positive constant depending only on ${\mathbf u}$,
and the integer $c_k$ is as in Theorem \ref{hirzthem} for $k\ne 0$ and $c_0 = 1$.
\end{lem}
\begin{proof}
Let $J$ be the complex structure of $F_2(\alpha)$ as a toric manifold.
It follows from \cite{cho-oh} section 7 and Theorem \ref{hirzthem} that
the mapping degree of the map
$$
ev_0 : \mathcal M_1(F_2(\alpha),L({\mathbf u});J;\beta_1+k\,[D_1]) \to L({\mathbf u})
$$
is $c_k$.
By our choice of $J_{S_0}$ the difference between $J$ and $J_{S_0}$ in
$C^k$ norm converges to zero as $S_0 \to \infty$.
(Here we equip the neck region $\cong [-S_0,S_0] \times S^{3}/\{\pm 1\}$
with a cylindrical metric and the other part with a metric independent of $S_0$ but depending only on
$e^{2S_0}\alpha$.)
We join $J$ with $J_{S_0}$ by a short path $\{J_{S_0,t}\}_{0 \leq t \leq 1}$.
\par
It suffices to show (see the proof of Theorem \ref{POMaslov 2}) that if $S_0$ is sufficiently large
then all the $J_{S_0,t}$-holomorphic discs bounding $L({\mathbf u})$ that have
energy $\le \beta_1\cap [\omega] + \delta_0({\mathbf u})$ have positive Maslov index.
We prove this by contradiction.
\par
Suppose to the contrary that we have $\alpha_i$, $S_i$ and
$v_i : (\Sigma_i,\partial \Sigma_i) \to (X(\alpha_i),L({\mathbf u}))$
such that
$S_i \to \infty$, $C^{-1} \le e^{2S_i}\alpha_i \le C$, $v_i$ is
$J_{S_i,t_i}$ holomorphic with energy $\le \beta_1\cap [\omega] + \delta_0({\mathbf u})$,
$v_i$ is non-constant and
the Maslov index of $[v_i]$ is non-positive. By choosing a subsequence,
we may assume $\lim_{i\to \infty} e^{2S_i}\alpha_i = \alpha'$ converges.
\par
We can use a compactness result such as the one in \cite{BEHWZ03}
(see also \cite{surgery}), and can take a subsequence with the following properties:
Consider $[-S_i, S_i] \times S^3/\{\pm 1\} \subset (F_2(\alpha_i),J_{S_i,t})$ and
cut $\Sigma_i$ along the dividing curve $v_i^{-1}(c_i \times S^3/\{\pm 1\})$
for a regular value  $c_i \in (0,1)$ of $s \circ v_i: \Sigma_i \to \R$.
($s : [-S_i,S_i] \times S^3/\{\pm 1\} \to \R$ is the projection to the first factor.)
Let $\Sigma_i^0$ be the part which is mapped to
$F_2(\alpha_i) \setminus \{(-\infty,c_i-S_i) \times S^3/\{\pm 1\} \cup C_{-2}\}$.
Then $v_i\vert_{\Sigma_i^0}$ converges in appropriate compact $C^{\infty}$
topology to a map $v_{\infty} : \Sigma_{\infty} \to F_2(0)_0$ where
$\Sigma_\infty$ conformally a disc with punctures.
\begin{rem}
We note that
$
\{c_i\} \times S^3/\{\pm 1\}
\subset [-S_i, S_i] \times S^3/\{\pm 1\} \subset F_2(\alpha_i)
$
is identified with $\{e^{c_i - S_i}\} \times S^3/\{\pm 1\}
\subset F_2(0)_0$ in the conical coordinate $r = e^s$. In fact we identified
$\{e^{-S_i}\} \times S^3/\{\pm 1\}
\subset F_2(0)_0$ with
$\{S_i\} \times S^3/\{\pm 1\}
\subset X(\alpha_i)$, which corresponds to $c_i = 0$, i.e.,
to $\{0\} \times S^3/\{\pm 1\} \subset F_2(\alpha_i)$.
\end{rem}
\par
We remark that image of the limit map $v_{\infty}$ can not be entirely contained
in the compact subset of $F_2(0)_0$.
This is because $F_0(0)_0$ does not contain
$J$-holomorphic disc with boundary on $L({\mathbf u})$ and of Maslov index non-positive.
\par
Then by choosing $\delta_0({\mathbf u})$ sufficiently small,
this implies that $\Sigma_{\infty}$ consists of one
component such that $\int_{\Sigma_{\infty}} v_{\infty}^*\omega = \beta \cap [\omega] = 1 - u_2$.
Namely it defines an element of $\mathcal M^\sharp_{(1,1)}(L({\mathbf u}); F_2(0))$.
\par
On the other hand, the first Chern classes of
$\R \times S^3/\{\pm 1\}$ and of $X(\alpha')$ are trivial.
Therefore the Maslov index of $v_i$ is $2$ for sufficiently large $i$,
a contradiction to the hypothesis that the Maslov index is non-positive.
\end{proof}
\begin{lem}\label{enddegrcount}
The mapping degree of (\ref{evat1}) is $1$.
\end{lem}
\begin{proof}
We consider $\mathcal M^\sharp(X(e^{2S_0}\alpha);0)$. It is easy to see that they consist
of fibers of the $\C$ vector bundle $X(e^{2S_0}\alpha) \to \C P^1$. Therefore
$ev^\sharp : \mathcal M^\sharp(X(e^{2S_0}\alpha);0) \to \mathcal R_1(\lambda)$ is a diffeomorphism.
It follows from Lemma \ref{degreeforv1}
that the degree $d_0$ of $ev_0 : \mathcal M_1(F_2(\alpha),L({\mathbf u});J_{S_0};\beta_1)
\to L({\mathbf u})$
is 1 if we choose $S_0$ sufficiently large. It then follows from Lemma \ref{gluing} that
\begin{equation} \text{\rm deg}\left[ev_0:\mathcal M^\sharp(X(e^{2S_0}\alpha);0) \,\,{}_{ev^\sharp}\times_{ev^\sharp}
\mathcal M^\sharp_{(1,1)}(L({\mathbf u});F_2(0)) \to L({\mathbf u})\right] \\
=  1.
\end{equation}
This proves Lemma \ref{enddegrcount}.
\end{proof}
\begin{lem}\label{53}
The degree of $ev^\sharp: \mathcal M^\sharp(X(\alpha);k) \to \mathcal R_1(\lambda)$ is $1$ if $k=0,1$ and is zero otherwise.
\end{lem}
\begin{proof}
We choose an almost complex structure of $X(\alpha)$ so that it is independent of $\alpha
\ne 0$.
\par
We also remark that for given $k, u$ there exist $\alpha$, $C$ and $S_0$
such that the assumption of Lemma \ref{degreeforv1} is satisfied.
In fact we first choose $C$ such that $2kC^{-1} \le \delta_0({\mathbf u})$, next
$S_0 > S_0(k,{\mathbf u},C)$ and then finally we choose $\alpha$ so that
$C^{-1} \le e^{2S_0} \alpha \le 2C^{-1} \le C$. Then
$k e^{2S_0}\alpha \le 2C^{-1}k \le \delta_0({\mathbf u})$ as required.
\par
Now
Lemmas  \ref{degreeforv1} and \ref{enddegrcount} imply the degree of
$ev_0 : \mathcal M_1((F_2(\alpha),L({\mathbf u}));J_{S_0};\beta_1 + k\, D_1) \to
L({\mathbf u})$ is $c_k$.
Lemma \ref{53} now follows from Lemma \ref{gluing}.
\end{proof}
We next consider a simultaneous resolution of the family $F_2(0;\epsilon)$.
Here $F_2(0;\epsilon)$ is obtained from $F_2(0)$ by deforming singularity
as in section \ref{Deformation}.  Note it is independent of $\epsilon$ as a symplectic manifold.
Here we consider both symplectic and almost complex structures. The latter depends on $\epsilon$.
Existence of a simultaneous resolution implies the existence of the family
$F_2(\alpha;\epsilon)$ parameterized by $\alpha$ and $\epsilon$ such that
for $(\alpha,\epsilon) \ne (0,0)$ it is smooth,
$F_2(0;\epsilon)$ is as  above and $F_2(\alpha;0) = F_2(\alpha)$. We note that
$F_2(0;\epsilon)$ is symplectomorphic to the smoothing $\widehat F_2(0)$ of $F_2(0)$
introduced in section \ref{Deformation}. We denote by
$$
S^2_{van}
$$
a vanishing cycle of $F_2(0;\epsilon)$, which is isotoped to the $(-2)$-curve in $X_0$ and
shrinks to the singular point of $F_2(0)$ as $\e \to 0$.
\par
Similarly we construct 2-parameter family of local models
$X(\alpha;\epsilon)$. We may assume that $X(\alpha;\epsilon)$ coincides with $X(\alpha;0) = X(\alpha)$ outside
compact set. We use this fact to define $\mathcal M^\sharp(X(\alpha,\epsilon);k)$
in the same way as $\mathcal M^\sharp(X(\alpha);k)$.
\begin{lem}\label{degind}
The degree of $ev^\sharp: \mathcal M^\sharp(X(\alpha,\epsilon);k) \to \mathcal R_1(\lambda)$ is independent of $(\alpha,\epsilon)
\ne (0,0)$.
\end{lem}
\begin{proof}
This can be proved by a standard cobordism argument using the fact that element of $\mathcal R_1(\lambda) \cong S^2$
represents a Reeb orbit with smallest action.
\end{proof}
\begin{cor}\label{55}
The degree of $ev^\sharp: \mathcal M^\sharp(X(0;\epsilon);k) \to \mathcal R_1(\lambda)$ is $1$ if $k=0,1$ and is 0 otherwise.
\end{cor}
This is immediate from Lemmas \ref{degind} and \ref{53}.
\begin{rem}
We remark that for $\epsilon =0$, $\alpha \ne 0$,  the elements of
$\mathcal M^\sharp(X(\alpha;0);0)$ are the fibers of the $\C$ bundle $X(\alpha;0) \to \C P^1$.
For $k\ne 0$ the moduli space $\mathcal M^\sharp(X(\alpha;0);k)$ consists of
those fibers together with sphere bubbles, which are (multiple cover of) the
zero section $\C P^1$. So it is nonempty for all $k \ge 0$. A direct counting of them
is rather cumbersome. We use Theorem \ref{hirzthem} and a
gluing argument to count them in Lemma \ref{53} and Corollary \ref{55}.
If $\epsilon \ne 0$ then there is no pseudo-holomorphic sphere in $X(0;\epsilon)$
since the symplectic form on $X(0;\epsilon)$ is exact.
Therefore $\mathcal M^\sharp(X(0;\epsilon);k)$, if non-empty, necessarily consists of
(proper) holomorphic maps from $\C$ without bubble.
\end{rem}
\par
We now need a slight modification of Lemma \ref{degreeforv1} to deal with
the moduli spaces associated with $(F_2(0;\epsilon),T(u))$ instead of $(F_2(\alpha),L(u,1-u))$.
\par
By gluing $X(0;\epsilon)$ with $F_2(0)_0$ in a similar way as before
we obtain $F_2(0;\epsilon)$ equipped with an almost complex
structures $J_{S_0,\epsilon}$ parameterized by $S_0$.
\par
$F_2(0;\epsilon)$ contains a cylindrical region $\cong [-S_0,  S_0] \times S^3/\{\pm 1\}$
on which $J_{S_0,\epsilon}$ is translation invariant. It also contains
$X(0;\epsilon) \setminus [S_0,\infty)\times S^3/\{\pm 1\}$ equipped with the symplectic form
$e^{-2S_0}\omega_{X(0,\epsilon)}$.
\par
Recall that when $L(u,1-u) \subset F_2(0)$ is considered as a submanifold in $\widehat{F}_2(0)
\cong F_2(0,\epsilon )$, we denote it by $T(u)$.
We define $\mathcal M_1(F_2(0;\epsilon),T(u);J_{S_0,\epsilon};\beta_1 + k\,  S^2_{van})$,
and $\mathcal M^{\#}(X(0;\epsilon);k)$  in the same way as
$\mathcal M_1(F_2(\alpha),L(u,1-u);\beta_1 + k\,\C P^1)$ and
$\mathcal M^{\#}(X(\alpha);k)$, respectively.

\begin{lem}\label{gluing5}
For each $k,\, u$ and $\epsilon$, there exists a constant $S_0(k,u,\epsilon)$ such that if $S_0 > S_0(k,u,\epsilon)$
the virtual fundamental cycle of the evaluation map
$$
ev_0:\mathcal M^{\sharp}(X(0;\epsilon);k) \,\,{}_{ev^\sharp}\times_{ev^\sharp} \mathcal M^\sharp_{(1,1)}(L(u,1-u); F_2(0)) \to L(u,1-u)
$$
defines the same homology class in $L(u,1-u) \cong T(u)$ as that of
$$
ev_0 : \mathcal M_1((F_2(0;\epsilon),T(u));J_{S_0,\epsilon};\beta_1 + k\, S^2_{van})\to T(u).
$$
\end{lem}
\begin{proof}
We remark that in Lemma \ref{gluing} we need an assumption
$C^{-1} \le e^{2S_0}\alpha$ since the `symplectic structure' of
$X(0)$ is degenerate. On the other hand,
the symplectic structure of $X(0;\epsilon)$ is non-degenerate.
Once this fact is understood the proof is the same as Lemma \ref{gluing}
and is now standard.
\end{proof}
\par
Using Lemmas \ref{enddegrcount}, \ref{gluing5}  and Corollary \ref{55} we finally prove the
following:
\begin{prop}\label{56}
There exists $S_1(u,\epsilon)$ depending only on $u$ and $\epsilon$ such that if
$S_0 > S_1(u,\epsilon)$,
then the degree of $ev_0 : \mathcal M_1(F_2(0;\epsilon),T(u);J_{S_0,\epsilon};
\beta_1 + k\, S^2_{van}) \to T(u)$ is $1$ if $k=0,1$ and is $0$ otherwise.
\end{prop}
\begin{proof}
We first show the following:
\begin{lem}\label{gibkempty}
There exists $k(\epsilon)$ such that  $\mathcal M^\sharp(X(0,\epsilon);k)
=\emptyset$ if $\vert k\vert > k(\epsilon)$.
\end{lem}
\begin{proof}
We remark $S^2_{van} \cap [\omega] = \alpha = 0$ in our case.
Therefore it follows that the Hofer energy of elements of
$$
\bigcup_{k} \mathcal M^\sharp(X(0,\epsilon);k)
$$
are independent of $k$ and constant (and so are bounded). The lemma now follows
from the Gromov-Hofer compactness.
\end{proof}
By Lemma \ref{gluing5}, the conclusion of Proposition \ref{56} follows
if  $\vert k\vert \le k(\epsilon)$ and
$$
S_0 \ge \max_{k; \vert k\vert \le k(\epsilon)}S_0(k,u,\epsilon).
$$
Therefore it is enough to prove that there exists a sufficiently large $S_0(u,\e) > 0$
such that
\begin{equation}\label{MMklarge}
\mathcal M_1(F_2(0;\epsilon),T(u);J_{S_0,\epsilon};\beta_1 + k\,S^2_{van}) = \emptyset
\end{equation}
for all $|k| > k(\e)$ if $S_0 > S_0(u,\e)$. We prove this statement by contradiction.
\par
Suppose to the contrary that there exists a sequence $S_i \to \infty$,
$k_i$ with  $\vert k_i\vert > k(\epsilon)$ and $(\Sigma_i, v_i,z_0) \in
\mathcal M_1(F_2(0;\epsilon),T(u);J_{S_i,\epsilon};\beta_1 + k_i\, S^2_{van})$.
\par
We first remark that $(\beta_1 + k\, S^2_{van}) \cap [\omega] = \beta_1 \cap [\omega]$
is independent of $k$. Therefore the symplectic energy of $v_i$ is uniformly bounded.
\par
We next consider $[-S_i, S_i] \times (S^3/\{\pm 1\}) \subset F_2(0;\epsilon)$.
We divide the domain of $v_i$ along $s = c_i$ for a regular value $c_i \in [S_i-1,S_i]$
in the same way as in the proof of Lemma \ref{gluing}. We consider the part that goes to
$X(0;\epsilon) \setminus (c_i,\infty) \times  S^3/\{\pm 1\}$,
which we denote by
$v_i^0 : \Sigma_i^0 \to X(0;\epsilon)$.
We claim that $(\Sigma_i^0,v_i^0)$ converges to an element of
$\mathcal M^\sharp(X(0,\epsilon);k)$ for some $k$ in compact $C^{\infty}$ topology,
after taking a subsequence.
\par
In fact we can use energy bound to find a subsequence such that
the restriction of $v_i$ to $v_i^{-1}( [S_i-1,S_i] \times S^3/\{\pm 1\})$
converges in $C^{\infty}$ topology. Therefore the action
$$
\int_{v_i^{-1}(\{c_i\} \times S^3/\{\pm 1\})} (\Theta\circ v_i)^*\lambda
$$
is uniformly bounded.
On the other hand, since $\alpha = 0$, the symplectic form $\omega$ is
exact on $X(0;\epsilon)$.
\par
Existence of a convergent subsequence of $(\Sigma^0_i,v_i^0)$
is then again a consequence of Gromov-Hofer compactness.
\par
Now by Lemma \ref{gibkempty} we have $\vert k_i\vert \le k(\epsilon)$
for sufficiently large $i$.
This contradicts to our assumption that $\vert k_i\vert > k(\epsilon)$.
\end{proof}
We next prove that only the moduli spaces for $k = 0, \, 1$ in Proposition \ref{56} contribute
to the potential function $\frak{PO}^{J_{S_0,\epsilon}}$.
Let $\beta_i \in H_2(F_2(\alpha),L(u_1,u_2);\Z)$ ($i=1,\dots,4$) be the
classes such that
$\mathcal M_1^{\text{\rm reg}}(F_2(\alpha),L(u_1,u_2);\beta_i) \ne \emptyset$
and $c_1 \cap \beta_i = 2$.
(Here `reg' means the moduli space of pseudo-holomorphic disks without
bubble.) \cite{cho-oh} Theorem 5.2 implies that there are exactly four homology classes satisfying this
condition. $\beta_1$ is the same class as before.
We denote by $\beta_i \in H_2(F_2(0;\epsilon),T(u);\Z)$ the class corresponding to them by
the obvious diffeomorphism $(F_2(0;\epsilon),T(u)) \cong (F_2(\alpha),L(u,1-u))$.
\par
\begin{lem}\label{nootherthingcontribute}
For each $E, \epsilon$ and $(u,1-u) \in \text{\rm Int}P(0)$ there exists $S_0(E,\epsilon,u)$
such that the following holds for $S_0 > S_0(E,\epsilon,u)$:
Suppose $\beta \in H_2(F_2(0;\epsilon),T(u);\Z)$ satisfies
\begin{subequations}\label{condforempty}
\begin{equation}
\beta \ne \beta_i, (i=1,2,3,4), \quad
\beta \ne \beta_1 + k\, S^2_{van}, k \in \Z,
\end{equation}
and
\begin{equation}
\beta\cap \omega \le E, \quad
\mu(\beta) = 2,
\end{equation}
\end{subequations}
where $\mu$ is the Maslov index. Then we have
$$
\mathcal M_1((F_2(0;\epsilon),T(u));J_{S_0,\epsilon};\beta) = \emptyset.
$$
\end{lem}
\begin{proof}
Suppose to the contrary that there exists a sequence $\beta_j
\in H_2(F_2(0;\epsilon),T(u);\Z)$
and
$$
(\Sigma_j,v_j,z_j) \in \mathcal M_1((F_2(0;\epsilon),T(u));J_{S_j,\epsilon};\beta_j)
$$
such that $\beta = \beta_j$ satisfies (\ref{condforempty}) and
$S_j \to \infty$. By the same reason as in the proof of Proposition \ref{56}, we have a
constant $C > 0$ such that
\begin{equation}\label{sigmajbdd1}
\int_{v_j^{-1}(\{c_j\} \times S^3/\{\pm 1\})} (\Theta\circ v_j)^*\lambda \le C
\end{equation}
for some regular value $c_j \in [S_j-1,S_j]$ of $s \circ v_j$.
(Here $[-S_j,S_j] \times S^3/\{\pm 1\} \subset F_2(0;\epsilon)$.)
This implies that
\begin{equation}\label{sigmajbdd2}
\int_{v_j^{-1}(\{d_j\} \times S^3/\{\pm 1\})} (\Theta\circ v_j)^*\lambda \le C
\end{equation}
where $d_j \in [-S_j,-S_j+1]$ is a regular value of $s \circ v_j$.
\par
We then consider the subsets $\Sigma_j^{\text{\rm int}}$ and $\Sigma_j^{\text{\rm out}}$
of $\Sigma_j$ which are mapped to
$X(0;\epsilon) \setminus (c_j,\infty) \times S^3/\{\pm 1\}$ and
$F_2(0;\epsilon) \setminus (-\infty,d_j) \times S^3/\{\pm 1\}$ respectively. We put
$\Sigma_j^{\text{\rm neck}} = \Sigma_j^{\text{\rm int}} \cap \Sigma_j^{\text{\rm out}}$
which is mapped into $[d_j,c_j] \times S^3/\{\pm 1\}$.
\par
The action bound (\ref{sigmajbdd1}) implies that $(\Sigma_j^{\text{\rm int}},v_j)$
converge to a pseudo-holomorphic curve
$(\Sigma_{\infty}^{\text{\rm int}},v_{\infty})$ in $X(0;\epsilon)$
in compact $C^{\infty}$ topology.
\par
The bound (\ref{sigmajbdd2}) implies that $(\Sigma_j^{\text{\rm out}},v_j,z_j)$
converge to a pseudo-holomorphic curve
$(\Sigma_{\infty}^{\text{\rm out}},v_{\infty},z_{\infty})$ in $F_2(0)_0$
in compact $C^{\infty}$ topology.
\par
Moreover $(\Sigma_{\infty}^{\text{\rm neck}},v_{\infty})$
converges to a union of pseudo-holomorphic maps
$(\Sigma_{\infty,c}^{\text{\rm neck}},v_{\infty,c})$
in $\R \times S^3/\{\pm 1\}$, such that
$(\Sigma_{\infty}^{\text{\rm int}},v_{\infty})$,
$(\Sigma_{\infty}^{\text{\rm out}},v_{\infty},z_{\infty})$,
and $(\Sigma_{\infty,c}^{\text{\rm neck}},v_{\infty,c})$
($c=1,\dots,K$) comprise the limit of $(\Sigma_j,v_j,z_j)$
in an appropriate `stable map compactification'.
\begin{rem}
Here we do not discuss the full details of the proof of the above stable map convergence
result which, for example, follows from the argument used in \cite{hofer,BEHWZ03}.
We only need this stable map convergence to calculate the index of the
linearized operator of Cauchy-Riemann equation of $(\Sigma_j,v_j,z_j)$
by summing over the indices of the components of the limit and to analyze
the indices of each component.
\end{rem}

Now we provide details of dimension counting arguments of various moduli
spaces we have considered.

The virtual dimension of the moduli space
containing $(\Sigma_{\infty}^{\text{\rm int}},v_{\infty})$ is $2$ or larger and
is $2$ only if this moduli space is $\mathcal M_1(X(0;\epsilon),k)$.
The virtual dimension of the moduli space
containing $(\Sigma_{\infty}^{\text{\rm out}},v_{\infty})$ is $2$ or larger
and is $2$ only if it is  $\mathcal M^\sharp_{(1,1)}(T(u); F_2(0))$.
\par
Moreover for the case of $\Sigma_{\infty,c}^{\text{\rm neck}}
\cong \R \times S^1$, the virtual dimension of the moduli space
containing  $(\Sigma_{\infty,c}^{\text{\rm neck}},v_{\infty,c})$
is greater than or equal to 2, and is $2$ only if the component
$(\Sigma_{\infty,c}^{\text{\rm neck}}, v_{\infty,c})$ is a trivial cylinder.
\par
The assumption $\mu(\beta_j) =2$ is equivalent to
$\dim \mathcal M_1(F_2(0;\epsilon),T(u));J_{S_j,\epsilon};\beta_j) = 2$ which
is possible only when $(\Sigma_{\infty,c}^{\text{\rm neck}},v_{\infty,c})$
is a trivial cylinder. For otherwise the translation along
the $[-S_j,S_j]$-direction of $[-S_j,S_j] \times S^3/\{\pm 1\} \subset
(F_2(0;\epsilon),J_{S_j,\epsilon})$ would provide an additional parameter.
\par
Therefore, $\mathcal M_1(F_2(0;\epsilon),L(u,1-u));J_{S_j,\epsilon};\beta_j)$
for a sufficiently large $j$ has the same dimension as that of the
fiber product over $S^2$ of the moduli spaces containing
$(\Sigma_{\infty}^{\text{\rm int}},v_{\infty})$ and
$(\Sigma_{\infty}^{\text{\rm out}},v_{\infty})$. This implies that
both moduli spaces have dimension 2.

Hence $\beta_j = \beta_1 + k_j\, S^2_{van}$
for some $k_j$ if $j$ is sufficiently large.
This is a contradiction.
\end{proof}
Theorem \ref{POsingptt} now follows from Proposition \ref{56} and
Lemma \ref{nootherthingcontribute}.
(See Theorem \ref{POcanwithb} and \eqref{PO^L} for the definition of $\frak{PO}$.)
\qed
\section{Proof of Theorem \ref{nonvanish}}\label{proof3}
In this section, we prove Theorem \ref{nonvanish}.
Note that the trivialization of the family $\bigcup_{a \in \C} X_a$ as smooth manifold
identifies the cohomology class $[D_1]$ in $X_0$ and $[S^2_{van}]$ in $X_a$.
Using the proof of Theorem \ref{POsingptt} in the last section,
each of the degree $1$ in Lemma \ref{55} contributes $1$ to the
coefficient $2$ in front of the last term of (\ref{POsingpt}).
The homology classes $\beta_1$ and $\beta_1+[S^2_{van}]$ satisfy the relations
\begin{equation}\label{intnumber}
\beta_1 \cap [S^2_{van}] = 1, \quad (\beta_1+[S^2_{van}])\cap [S^2_{van}] = -1.
\end{equation}
Now we consider the cohomology class
$$
\frak b = T^{\rho}PD[S^2_{van}] \in H^2(\widehat F_2(0),\Lambda_+).
$$
Then (\ref{intnumber}) and
Theorem \ref{POsingptt} imply that the potential function with bulk,
$\frak{PO}^{\frak b}$, becomes
\begin{equation}\label{PObulk}
\frak{PO}^{\frak b}
= T^{u_1}y_1 + T^{u_2}y_2 +T^{2-u_1-2u_2}y_1^{-1}y_2^{-2} +
(e^{T^{\rho}}+ e^{-T^{\rho}})T^{1 -u_2}y_2^{-1}.
\end{equation}
(See \cite{toric2} Theorem 3.4.)
Now let $u_1, u_2$ be as in Theorem \ref{nonvanish}.
We put
\begin{equation}\label{rhoichi}
2\rho = u_2 - u_1 = u_2 - (1-u_2) = 2u_2 - 1
\end{equation}
and consider (\ref{PObulk}) at $u=(u_1, u_2)$ for some $\rho$. It becomes
\begin{equation}
T^{u_1}(y_1 + y_1^{-1}y_2^{-2} + 2y_2^{-1})
+ T^{u_1+2\rho}(y_2 + y_2^{-1})
+ (\text{\rm higher order})y_2^{-1}.
\end{equation}
The condition for $(y_1,y_2)$ being its critical point is
\begin{eqnarray}
0 &=& 1 - y_1^{-2}y_2^{-2} \label{crit41}\\
0 &=& -2y_1^{-1}y_2^{-3} - 2y_2^{-2} + T^{2\rho}(1-y_2^{-2})
+ (\text{\rm higher order})y_2^{-2}. \label{crit42}
\end{eqnarray}
We consider the solution $y_1y_2 =-1$ of (\ref{crit41}).
Then (\ref{crit42}) has a solution
$y_2 = \pm 1 + (\text{\rm higher order})$.
The proof of Theorem  \ref{nonvanish} is complete.
\qed
\begin{rem}
The result of this paper,
combining with the study of spectral invariants, also implies that
among the 4 quasi-morphisms
$$
Ham(S^2(1) \times S^2(1)) \to \R
$$
obtained by \cite{entov-pol06}
using quantum-cohomology (without
bulk deformations), two of them
are different from the other two.
This fact and the following theorem will be proven in a forthcoming paper \cite{toric4}.
(We use spectral invariants with bulk deformation to prove Theorem \ref{main1}.)
\end{rem}
\begin{theorem}\label{main1}
Let $\widetilde{Ham}(S^2(1) \times S^2(1))$ be the
universal cover of the group of Hamiltonian diffeomorphisms
of $S^2(1) \times S^2(1)$.
Then there exists an infinitely many
Calabi quasi-morphisms $\widetilde{Ham}(S^2(1) \times S^2(1))
\to \R$ such that any finite subset thereof is linearly independent.
\end{theorem}
\begin{rem}\label{int-result}
The statement that Fukaya category of a toric manifold is split-generated by
(strongly) balanced torus fibers,
which we will prove in a future article with M. Abouzaid, implies that
$\varphi(T(u)) \cap (S^1_{\text{\rm eq}} \times S^1_{\text{\rm eq}}) \ne \emptyset$ for
any Hamiltonian diffeomorphism.
This intersection statement also follows from the theory of spectral invariant
with bulk deformation.
We will discuss these points in a future article \cite{toric4}.
\end{rem}
\noindent
\section{Appendix: $\mathfrak m_0(1)$ in the canonical model.}\label{secapend}
\par
\bigskip
In section 5.4 of \cite{fooo-book}, we gave a construction of the canonical model
for the filtered $A_{\infty}$-algebra (see \cite{fooo-can} for a concise exposition).
In \cite{fooo-book}, we first constructed a filtered $A_{\infty}$-algebra structure
$\{\mathfrak m_k\}$ on a certain subcomplex $C(L;\Lambda_0)$ of the smooth
singular chain complex $S(L;\Lambda_0)$ with $\Lambda_0$-coefficients.
Then using the argument of summation over certain rooted decorated
trees, we obtained its canonical model, i.e., a filtered $A_{\infty}$ algebra structure $\{\mathfrak m_k^{\text{\rm can}}\}$ on
the classical cohomology $H(L;\Lambda_0)$ with $\Lambda_0$-coefficients.
\par
Here we  only give the definition of the set $G^+_{k}$ of rooted decorated trees used
in the construction of $\mathfrak m_0^{\text{\rm can}}$.
(See (5.4.31) in \cite{fooo-book} or section 3 of \cite{fooo-can} for details.)
The set $G^+_{k}$ consists of  quintets $(T,i,v_0,V_{\text{\rm tad}},\eta)$:
\begin{enumerate}
\item $T$ is a tree with an embedding $i$ to the unit disc.
\item The set of vertices of valency 1 is the disjoint union of the set
$C^0_{\text{\rm ext}}(T)$ of exterior vertices and $V_{\text{\rm tad}}$, tad poles. We set  $k= \# C^0_{\text{\rm ext}}(T)$.
\item A root vertex $v_0$ is an element of $C^0_{\text{\rm ext}}(T)$.
\item  Denote by $C^0_{\text{\rm int}}(T)$ the union of the set of vertices of valency at least 2 and $V_{\text{\rm tad}}$.
The set $C^0_{\text{\rm int}}$ is equipped with $\eta$, which assign a class in $H_2(M,L;\Z)$ to
each element of $C^0_{\text{\rm int}}$.
\item For each vertex $v \in C^0_{\text{\rm int}}(T)$ with $\eta(v) \neq 0$, the valency at $v$ is at least 3.
\end{enumerate}

In this appendix, we summarize basic properties of the potential function for the canonical model
under Assumption \ref{Maslov 2} below.
Namely, we prove the following Theorems \ref{POcan} and \ref{POcanwithb}.
Let $M$ be a symplectic manifold and $L$ be a relatively
spin Lagrangian submanifold.
\begin{ass}\label{Maslov 2}
If $\beta \in \pi_2(M,L)$ is nonzero and
$\mathcal M_{k+1}(\beta) \ne \emptyset$,
then $\mu_L(\beta) \ge 2$.
\end{ass}
Typical examples of Lagrangian submanifold satisfying this assumption
are
\begin{enumerate}
\item[(a)] Lagrangian torus fibers of Fano toric manifolds \cite{cho-oh}, \cite{toric1},
\item[(b)]monotone Lagrangian submanifolds with minimal Maslov number $2$,
\item[(c)]$\dim L =2$ and the almost complex structure is generic.
\end{enumerate}
Let $L$ be a relatively spin Lagrangian submanifold such that it satisfies Assumption \ref{Maslov 2}
and $\beta \in \pi_2(M,L)$ with  $\mu_L(\beta) = 2$.
In this case it is easy to see
that the moduli space $\mathcal M_{1}(\beta)$ has no boundary
(in the sense of Kuranishi structure) and
$\dim \mathcal M_{1}(\beta) = n$. So the virtual
fundamental cycle
\begin{equation}
ev_{0*}([\mathcal M_{1}(\beta)]) \in H_n(L;\Z) \cong \Z
\label{23.76}
\end{equation}
is well-defined. (It is an integral class since $\mathcal M_1(\beta)$
has no sphere components and so has trivial automorphism group.)
\begin{theorem}\label{POcan}
Let $(M,L)$ satisfy Assumption \ref{Maslov 2} and consider the canonical model
$(H(L;\Lambda_{0,\text{\rm nov}}),\mathfrak m^{\text{\rm can}})$
of the filtered $A_{\infty}$ algebra in Theorem A of \cite{fooo-book}. Then we have
the deformed
$$
\mathfrak m_0^{\text{\rm can}}(1) = \sum_{\mu_L(\beta) = 2}
e T^{\omega(\beta)}
ev_{0*}([\mathcal M_{1}(\beta)]).
$$
Here right hand side is an element of
$H_n(L;\Lambda_{0,\text{\rm nov}}) \cong H^0(L;\Lambda_{0,\text{\rm nov}})$.
\end{theorem}
\begin{rem}\label{Remark 23.78}
We remark that Assumption \ref{Maslov 2} in general depends on the
almost complex structure $J$ of $M$. If we fix
$J$ satisfying Assumption \ref{Maslov 2}, then
$\mathfrak m_0^{\text{\rm can},J}(1)$ is well-defined and
is independent of the other choices.
If we have $J$ and $J'$ both of which satisfy Assumption \ref{Maslov 2},
then $\mathfrak m_0^{\text{\rm can},J}(1) = \mathfrak m_0^{\text{\rm can},J'}(1)$
if $J$ can be joined to $J'$ by a path of compatible
almost complex structures satisfying Assumption \ref{Maslov 2}.
\par
This follows from Theorem \ref{POcan} and the fact that
the homology class of the virtual fundamental cycle
(\ref{23.76}) is well-defined in the above sense.
\end{rem}
\begin{proof}
We consider the filtered $A_{\infty}$ structure
$$
\mathfrak m_k : B_kC(L;\Lambda_{0,\text{\rm nov}})[1] \to C(L;\Lambda_{0,\text{\rm nov}})[1]
$$
on the countably generated subcomplex $C(L;\Lambda_{0,\text{\rm nov}})$
of smooth singular chain complex.
By definition we have
$$
\mathfrak m_{k,\beta} = 0
$$
unless $\beta=\beta_0$ or $\mu_L(\beta) \ge 2$.
\par
Now starting from the filtered $A_{\infty}$-structure $\{\mathfrak m_k\}$, we construct the
filtered $A_{\infty}$ structure $\{\mathfrak m_0^{\text{\rm can}}\}$ on $H(L;\Lambda_{0,\text{\rm nov}})$
by the method of \S 5.4.4 in \cite{fooo-book}.
Then we have
$$
\mathfrak m^{\text{\rm can}}_{k}(1) = \sum_{\Gamma} \mathfrak m_{\Gamma}(1)
$$
where $\Gamma$ runs over the set $G^+_{1}$.
\par
By definition, each element of $G^+_{1}$ has no
exterior edge other than the root vertex $v_0$.   Hence
all vertices other than $v_0$ with one edge must be a tad pole.
There must be at least one tad pole.
If the sum of Maslov indices of the class
assigned to the vertices is greater than $2$,
then $\mathfrak m_{\Gamma}(1)$ for such a $\Gamma$
is zero, since it will be an element of
$H_{n+\mu_L(\beta) - 2}$. (Here $\beta$ is the
sum of the all homotopy class assigned to the
vertices.)
\par
Therefore if $\mathfrak m_{\Gamma}(1)$ is non-zero then
there can be only one vertex (that is the tad pole) and
no other vertices. Namely we have obtained
$$
\mathfrak m_0^{\text{\rm can}}(1) = \sum_{\mu_L(\beta) = 2}
e T^{\omega(\beta)}
ev_{0*}([\mathcal M_{1}(\beta)]).
$$
\end{proof}

Under Assumption \ref{Maslov 2}, any $b = \sum x_i {\mathbf e}_i$, where
$x_i \in \Lambda_0$ of degree 0 and $\{\mathbf{e}_i\}$ is the basis
of $H^1(L;\Z)$ chosen before,
is a solution of the (weak) Maurer-Cartan equation.

Then the following theorem can be proved by the same
strategy of \cite{toric1} section 11.

\begin{theorem}\label{POcanwithb}
Let $(M,L)$ satisfy Assumption \ref{Maslov 2}. we have
$$
\mathfrak m_0^{\text{\rm can},b}(1) = \sum_{\mu_L(\beta) = 2}
e T^{\omega(\beta)} \exp{(b \cap \partial \beta)}
ev_{0*}([\mathcal M_{1}(\beta)]).
$$
Here $b \cap \partial \beta = \sum x_i (\mathbf{e}_i \cap \partial \beta) \in \Lambda_0$.
\end{theorem}

We remark
\begin{equation}\label{PO^L}
\frak{PO}^L(b) = \mathfrak m_0^{\text{\rm can},b}(1) \cap [L]/e.
\end{equation}
We note by definition that
$$
ev_{0*}([\mathcal M_{1}(\beta)]) \cap [L] =\text{\rm deg}\left[ev_{0*}:\mathcal M_{1}(\beta)
\to L \right].
$$
Therefore, using the same argument in Remark 7.2, we have the following
\begin{theorem}\label{POMaslov 2}
Let $\{J_t \}_{0\leq t \leq 1}$ be a family of tame almost complex structures such that
Assumption \ref{Maslov 2} holds for all $J_t$.
Then we have
$\mathfrak{PO}^{L,J_0} = \mathfrak{PO}^{L,J_1}$.
\end{theorem}
\begin{rem}\label{welldefrem}
In Theorem \ref{POMaslov 2} we put a rather strong
hypothesis that Assumption \ref{Maslov 2} holds for all $J_t$.
This assumption is satisfied for the family of almost complex
structures used in the first proof of Theorem \ref{hirzthem} in
section \ref{proof2}.
\par
In our two dimensional case,
Assumption \ref{Maslov 2} holds for generic $J$ but there may
exist a codimension one set of $J$ for which Assumption \ref{Maslov 2}
is not satisfied.
As is shown in \cite{fooo-book} the potential function
$\frak{PO} : \mathcal M_{\text{\rm weak}}(L) \to \Lambda_0$ is
well-defined. In our two dimensional case,
dimension counting implies that
$\sum_{k=0}^{\infty}\frak m_k(b^{\otimes k}) \in H^0(L;\Lambda_0)$
for any $b \in H^1(L;\Lambda_0)$.
(This is a consequence of Assumption \ref{Maslov 2}.)
Therefore we have an isomorphism $H^1(L;\Lambda_0) \to \mathcal M_{\text{\rm weak}}(L)$.
However this isomorphism depends on the choice of almost complex structure.
Therefore if we regard $\frak{PO}$ as a function
$H^1(L;\Lambda_0) \to \Lambda_0$ it is well-defined up to
coordinate change congruent to the identity modulo $\Lambda_{0,\text{\rm nov}}^+$.
We however remark that for the proof of Theorem \ref{maintheorem}
and other applications in this paper, it is
enough to calculate  $\frak{PO}$ modulo coordinate change.
\end{rem}


\begin{thebibliography}{FOOO7}
\bibitem[AF]{A-F} P. Albers and U. Frauenfelder,
{\em  A nondisplaceable Lagrangian torus in $T^*S^2$},
Comm. Pure Appl. Math. 61 (2008), 1046-1051.
\bibitem[Au]{auroux}
D. Auroux,
{\em Special Lagrangian fibrations, wall-crossing, and mirror symmetry},
to appear in Surveys in Differential Geometry,
arXiv:0902.1595.
\bibitem[BEHWZ]{BEHWZ03} F. Bourgeois, Y. Eliashberg, H. Hofer,
K. Wysocki, and E. Zehnder,
{\em Compactness results in symplectic sield theory},
Geom. Topol.  7  (2003)
799-888.
\bibitem[C]{Cho} C.-H. Cho, Non-displaceable Lagrangian submanifolds and
Floer cohomology with non-unitary line bundle, J. Geom. Phys. 58 (2008), 213-226.
\bibitem[CO]{cho-oh} C.-H. Cho and Y.-G. Oh, {\em Floer cohomology and
disc instantons of Lagrangian torus fibers in Fano toric manifolds},
Asian J. Math. 10 (2006), 773--814.
\bibitem[EP]{entov-pol06} M. Entov and L. Polterovich,
{\em Quasi-states and symplectic intersections},
Comment. Math. Helv. 81 (2006), 75--99.
\bibitem [FOOO1]{fooo-book} K. Fukaya, Y.-G. Oh, H. Ohta and K. Ono,
Lagrangian Intersection Floer theory -- Anomaly and Obstructions --,
AMS/IP Studies in Advanced Mathematics, vol 46, 2009,
Amer. Math. Soc./International Press.
\bibitem[FOOO2]{fooo-can} K. Fukaya, Y.-G. Oh, H. Ohta and K. Ono,
{\em Canonical models of filtered $A_{\infty}$-algebras and Morse complexes},
In: New Perspectives and Challenges in Symplectic Field Theory, CRM Proceedings
and Lecture Notes 49, 201-227, Amer. Math. Soc. 2009.
\bibitem [FOOO3]{toric1} K. Fukaya, Y.-G. Oh, H. Ohta and K. Ono,
{\em Lagrangian Floer theory on compact toric manifolds I,}
Duke. Math. J. 151 (2009), 23-174.
\bibitem [FOOO4]{toric2} K. Fukaya, Y.-G. Oh, H. Ohta and K. Ono,
{\em Lagrangian Floer theory on compact toric manifolds II: bulk deformation,}
submitted, arXiv:0810.5774.
\bibitem [FOOO5]{toric3} K. Fukaya, Y.-G. Oh, H. Ohta and K. Ono,
{\em Lagrangian Floer theory and mirror symmetry on compact toric manifolds,}
in preparation.
\bibitem [FOOO6]{surgery} K. Fukaya, Y.-G. Oh, H. Ohta and K. Ono,
{\em Lagrangian surgery and metamorphosis of pseudo-holomorphic polygons,}
in preparation. (Its preliminary version is available at
http://www.math.kyoto-u.ac.jp/$\sim$fukaya/ as ``Chapter 10'' of \cite{fooo-book}.)
\bibitem [FOOO7]{toric4} K. Fukaya, Y.-G. Oh, H. Ohta and K. Ono,
{\em Spectral invariants with bulk, quasi-morphisms and Lagrangian Floer theory,}
in preparation
\bibitem[Hi]{hind} R. Hind, {\em Lagrangian spheres in $S^2 \times S^2$,} Geom. Funct. Anal.,
14 (2004), 303-318.
\bibitem[H]{hofer} H. Hofer, {\em Pseudoholomorphic curves in symplectizations
with applications to the Weinstein conjecture in dimension three}, Invent. Math.,
114 (1993), 515 - 563.
\bibitem[LO]{L-O} H.-V. Le and K. Ono,
{\em Parameterized Gromov-Witten invariants and topology of symplectomorphism
groups}, Adv. Studies Pure Math., 52, 2008, 51-75.
\bibitem[LR]{LR}
A.-M. Li and Y. Ruan,
{\em Symplectic surgery
and Gromov-Witten invariants of Calabi-Yau 3-folds}, Invent.
Math. 145 (2001), 151-218.
\bibitem[Mc1]{Mc} D. McDuff,
{\em Examples of symplectic structures}, Invent. Math. 89 (1987), 13-36.
\bibitem[Mc2]{MD} D. McDuff,
{\em The structure of rational and ruled symplectic $4$-manifolds}, Journ. Amer. Math. Soc. 3
(1990), 679-712.
 \bibitem[MK]{MK} J. Morrow and K. Kodaira, Complex Manifolds,
 Holt, Rinehart, Winston, New York 1971.
 \bibitem [NNU1]{nnu1}
T. Nishinou, Y. Nohara, and K. Ueda,
{\em Toric degenerations of Gelfand-Cetlin systems and potential functions},
preprint,
 arXiv:0810.3470.
\bibitem [NNU2]{nnu2}
T. Nishinou, Y. Nohara, and K. Ueda,
{\em Potential functions via toric degenerations},
preprint,
arXiv:0812.0066.
\end{thebibliography}
\end{document}